\documentclass[12pt]{amsart}

\usepackage{amssymb,mathrsfs,amsmath,amsthm,color,bm,mathtools,bbm,wasysym,subfig,float}
\usepackage[noadjust]{cite}

\usepackage{enumerate}
\usepackage[centering]{geometry}
\allowdisplaybreaks

\theoremstyle{plain}
\newtheorem{thm}{Theorem}[section]
\newtheorem{defn}[thm]{Definition}

\newtheorem{prop}[thm]{Proposition}
\newtheorem{cor}[thm]{Corollary}
\newtheorem{lem}[thm]{Lemma}

\theoremstyle{remark}
\newtheorem{rem}{Remark}
\numberwithin{equation}{section}

\DeclareMathOperator{\hdim}{\dim_H}

\newcommand{\Q}{\mathbb Q}
\newcommand{\N}{\mathbb N}
\newcommand{\R}{\mathbb R}

\newcommand{\lm}{\mathcal L}
\renewcommand{\hm}{\mathcal H}
\newcommand{\hc}{\mathcal H_\infty}
\newcommand{\ce}{\mathcal E}

\newcommand{\ck}{\mathcal K}

\newcommand{\cb}{\mathcal B}

\newcommand{\ca}{\mathcal A}

\newcommand{\scg}{\mathscr G}

\newcommand{\bx}{\mathbf{x}}
\newcommand{\ba}{\mathbf{a}}
\newcommand{\by}{\mathbf{y}}
\newcommand{\bz}{\mathbf{z}}

\newcommand{\qaq}{\mathrm{\quad and\quad}}

\newcommand{\be}{\bm\epsilon}

\begin{document}
		\title[Dichotomy laws for the Hausdorff measures of shrinking target sets]{Dichotomy laws for the Hausdorff measure of shrinking target sets in $\beta$-dynamical systems}
	\author{Yubin He}

	\address{Department of Mathematics, Shantou University, Shantou, Guangdong, 515063, China}

	\email{ybhe@stu.edu.cn}

%
%

	\subjclass[2020]{11J83, 11K60}

	\keywords{shrinking targets, Hausdorff measure, dichotomy law, Diophantine approximation.}
	\begin{abstract}
		In this paper, we investigate the Hausdorff measure of shrinking target sets in $\beta$-dynamical systems. These sets are dynamically defined in analogy to the classical theory of weighted and multiplicative approximation. While the Lebesgue measure and Hausdorff dimension theories for these sets are well-understood, the Hausdorff measure theory in even one-dimensional settings remains unknown. We show that the Hausdorff measure of these sets is either zero or full depending upon the convergence or divergence of a certain series, thus providing a rather complete measure theoretic description of these sets.
	\end{abstract}
	\maketitle
\section{Introduction}\label{s:intro}

The central question in Diophantine approximation is: how well can a given
real number $x\in[0,1)$ be approximated by rational numbers and various generalisations thereof. For a positive function $\psi:\N \to \R^+$, the classical set $W(\psi)$ of \emph{$\psi$-approximable points} is defined as the points $x\in[0,1)$ such that $|x-p/q|<\psi(q)$ for infinitely many $p/q\in\Q$.
For monotonic $\psi$,
Khintchine Theorem \cite{Kh24}  and Jarn\'ik Theorem \cite{Jarnik31} provide beautiful and strikingly zero-full dichotomies for the ``size" of such sets expressed in terms of, respectively, Lebesgue and Hausdorff measures. Building on the improvements on these two theorems in \cite{BDV06,BFV16}, their modern versions can be combined into the following unifying statement. Throughout,  a continuous, nondecreasing function defined on $\R_{\ge 0}$ and satisfying $f(0)=0$ will be called a dimension function. For a dimension function $f$, the Hausdorff $f $-measure of $E$ is denoted by $\hm^f(E)$. When $f(r)=r^s$, we write $\hm^s$ in place of $\hm^f$.

\begin{thm}[Khintchine-Jarn\'ik Theorem]\label{t:KJ}
	Let $\psi:\N\to\R^+$ be a positive function and let $f$ be a dimension function such that $r^{-1}f(r)$ is decreasing. Then,
	\[\hm^f\big(W(\psi)\big)=\begin{cases}
		0&\text{if $\sum_{q=1}^{\infty}qf\big(\psi(q)\big)<\infty$},\\
		\hm^f([0,1))&\text{if $\sum_{q=1}^{\infty}qf\big(\psi(q)\big)=\infty$ and $\psi$ is monotonic}.
	\end{cases}\]
\end{thm}
\noindent Within this classical setup, Theorem \ref{t:KJ} provide a complete measure theoretic description of the set of $\psi$-approximable points.

In 1995, Hill and Velani \cite{HiVe95} introduced a natural analogue of the classical $\psi$-approximable points from Diophantine approximation to dynamical systems, with the aim of studying the approximation properties of orbits. More precisely, consider a transformation $ T $ on a metric space $ (X,d) $. Let $ \{E_n\} $ be a sequence of subsets of $X$ with diameter going to $0$ as $ n\to\infty $. The \emph{shrinking target set} is defined by
\[W(T,\{E_n\}):=\{x\in X:T^nx\in E_n\text{ for i.m.\,$n\in\N$}\},\]
where i.m.\,stands for {\em infinitely many}.
Since their initial introduction, these sets have been investigated in various dynamical systems, with numerous authors contributing to the study of their size, expressed in terms of Lebesgue measure, Hausdorff dimension and Hausdorff measure. To name but a few, see  \cite{BK24,KZ23,LLVZ23,Phi67} for Lebesgue measure, \cite{BarRa18,HiVe95,HiVe97,HiVe99,KLR22,LLVZ23,LWWX14,SW13,Wang18,WZ21,YW23} for Hausdorff dimension and \cite{AlBa21,BK24,HiVe02,LSV07} for Hausdorff measure. The results for Hausdorff measure are considerably less developed than those for Hausdorff dimension, as Hausdorff measure can be
regarded as much ``finer" than Hausdorff dimension, with a corresponding more intricate  proof.

In this paper, we are particularly interested in the Hausdorff measure of the shrinking target sets in $\beta$-dynamical systems, aiming to establish dichotomy laws akin to  Khintchine-Jarn\'ik Theorem.

Let us start with introducing some necessary notation and definitions. By $a\ll b$ we mean there is an unspecified constant $c$ such that $a\le cb$. By $a\asymp b$ we mean $a\ll b$ and $b\ll a$. In the current work, we  endowed the unit hypercube $ [0,1]^d $ with the maximum norm $ |\cdot| $, i.e.\,for any $ \bx=(x_1,\dots, x_d)\in [0,1]^d $, $ |\bx|=\max\{|x_1|,\dots,|x_d|\} $. Therefore, a ball is corresponding to a Euclidean hypercube. The $d$-dimensional Lebesgue measure is denoted by $\lm^d$.
For two dimension functions $f$ and $g$, by $f\preceq g$
 we mean that
\begin{equation}\label{eq:prec}
	\frac{f(y)}{g(y)}\le \frac{f(x)}{g(x)}\quad\text{for any $0<x<y$},
\end{equation}
where the unspecified constant does not depend on $x$ and $y$. Equation \eqref{eq:prec} implies that
\[\lim_{r\to 0^+}\frac{f(r)}{g(r)}>0.\]
The limit could be finite or infinite. If the limit is infinite, then we write $f\prec g$ for $f\preceq g$. Moreover, if $f(r)=r^s$ (respectively $g(r)=r^s$), then we simply write $s\preceq g$ (respectively $f\preceq s$) for $f\preceq g$.

For $\beta>1$, the $\beta$-transformation $T_\beta:[0,1)\to[0,1)$ is defined by
\[T_\beta x=\beta x\ (\textrm{mod}\ 1).\]
Analogous to the classical theory of multiplicative and weighted Diophantine approximation, the following two types of sets have been introduced in $\beta$-dynamical systems. Let $d\ge 1$ and  $1<\beta_1\le\beta_2\le\cdots\le\beta_d$. Let $\bm h=(h_1,\dots,h_d)$ be a $d$-tuple of Lipschitz functions defined on $[0,1)^d$ with Lipschitz constant $L\ge 0$, i.e.
\[|h_i(x)-h_i(y)|\le L|x-y|\quad\text{for $x,y\in[0,1)$}.\]
Let $\Psi=(\psi_1,\dots,\psi_d)$ be a $d$-tuple of functions with $\psi_i:\N\to \R^+$ for $1\le i\le d$, and $\psi:\N\to \R^+$ be a positive function.
 Define
\begin{equation}\label{eq:weighted}
	W_d(\Psi,\bm h):=\big\{\bx\in[0,1)^d:|T_{\beta_i}^nx_i-h_i(x_i)|<\psi_i(n)\text{ $(1\le i\le d)$ for i.m.\,$n$}\big\}
\end{equation}
and
\begin{equation}\label{eq:multiplicative}
	W_d^\times(\psi,\bm h):=\bigg\{\bx\in[0,1)^d:\prod_{i=1}^{d}|T_{\beta_i}^nx_i-h_i(x_i)|<\psi(n)\text{ for i.m.\,$n$}\bigg\}.
\end{equation}
 When $d=1$, $W_d(\Psi,\bm h)$ and $W_d^\times(\psi,\bm h)$ are defined exactly in the same way, and we write $W_1(\psi,h)$ for brevity. Here, the $d$-tuple of Lipschitz functions $\bm h$ enables us to handle the shrinking target sets and their variations in a unified manner. For example, if $h_i(x_i)\equiv a_i\in[0,1)$ for $1\le i\le d$, then $W_d^\times(\psi,\bm h)$ and $W_d(\Psi,\bm h)$ are the classical shrinking target sets. On the other hand, if $h_i(x_i)=x_i$ for $1\le i\le d$, then $W_d^\times(\psi,\bm h)$ and $W_d(\Psi,\bm h)$ are refered to as {\em recurrence sets}.

The Lebesgue measure and dimensional theories of these sets  are well-established through a series of studies. We refer to \cite{HLSW22,KZ23,LLVZ23,Phi67} for the Lebesgue measure and \cite{He24b,HL24,LLVWZ24,LLVZ23,TW14,Wang18,YW23} for the Hausdorff dimension.
However, the Hausdorff measure result remains unknown in the literature, even in the one-dimensional setting. The main results of the present paper are the following theorems, which characterize the Hausdorff $f$-measure of the sets $W_d(\Psi,\bm h)$ and $W_d^\times(\psi,\bm h)$.
\begin{thm}\label{t:rectangle}
	Let $1<\beta_1\le\cdots\le \beta_d$. Suppose that each $\beta_i$ is an integer. Let $f$ be a dimension function such that $f\preceq d$, and that for each $1\le k\le d-1$, either $k\preceq f$ or $f\preceq k$. Then,
		\[\hm^f\big(W_d(\Psi,\bm h)\big)=\begin{cases}
		0&\text{ if $\sum_{n=1}^{\infty}s_n(\Psi,f)\prod_{i=1}^{d}\beta_i^n<\infty$},\\
		\hm^f([0,1]^d)&\text{ if $\sum_{n=1}^{\infty}s_n(\Psi,f)\prod_{i=1}^{d}\beta_i^n=\infty$},
	\end{cases}\]
	where
	\[s_n(\Psi,f)=\min_{\tau\in\ca_n}\biggl\{f(\tau)\prod_{i\in\ck_{n,1}(\tau)}\frac{\beta_i^{-n}}{\tau}\prod_{i\in\ck_{n,2}(\tau)}\frac{\beta_i^{-n}\psi_i(n)}{\tau}\biggr\},\]
	and, in turn $\ca_n=\{\beta_1^{-n},\dots,\beta_d^{-n},\beta_1^{-n}\psi_1(n),\dots,\beta_d^{-n}\psi_d(n)\}$,
	\[\ck_{n,1}(\tau):=\{i:\beta_i^{-n}\le \tau\}\quad\text{and}\quad\ck_{n,2}(\tau):=\{i:\beta_i^{-n}\psi_i(n)\ge \tau\}.\]
\end{thm}
\begin{rem}
	Note that the classical set $W(\psi)$ is a  $\limsup$ set of balls. The mass transference principle from balls to balls by Beresnevich and Velani \cite{BV06} allows us to transfer Lebesgue measure statements for a $\limsup$ set of balls (Khintchine Theorem) to Hausdorff measure statements for the shrunk $\limsup$ set of balls (Jarn\'ik Theorem). That is, Khintchine Theorem indeed implies Jarn\'ik  Theorem. For the set $W_d(\Psi,\bm h)$, by Lemma \ref{l:length}, it can be almost expressed as a $\limsup$ set defined by hyperrectangles. Although the Lebesgue measure for $W_d(\Psi,\bm h)$ is already established, little is known about its Hausdorff measure, since there is currently no general principle `from rectangles to rectangles' that allows us to transfer Lebesgue measure statements for a $\limsup$ set of hyperrectangles to Hausdorff measure statements for the corresponding shrunk $\limsup$
	set of hyperrectangles. Besides, establishing such a principle appears challenging unless stronger conditions are imposed. As explained in \cite[\S 4]{WW21}, one of the major challenges is the necessity of considering an optimal covering for the entire collection of hyperrectangles that define the $\limsup$ set, rather than treating each one individually. In \cite[Theorem 3.2]{WW21} and \cite[Theorems 2.8 and 2.9]{He24a}, the authors showed that $\hm^s(W_d(\Psi,\bm h))=\hm^s([0,1]^d)=\infty$ with $s$ as the Hausdorff dimension of $W_d(\Psi,\bm h)$, based on the fact, expressed in our setting as
	\[\text{$s_n(\Psi,f)\prod_{i=1}^{d}\beta_i^n\ge c>0$ for i.m.\,$n$}.\]
	Obviously, the divergence of the series in Theorem \ref{t:rectangle} does not guarantee this properties, indicating that their approach is not fully applicable to our setting. To overcome this barrier, we combine their method with a new one based on some geometric observations (see Section \ref{ss:observation}).
 Our method cannot be extended to cases where  $\beta$ is non-integer. As discussed above, addressing the optimal cover for a collection of hyperrectangles potentially requires that these hyperrectangles be well-distributed, a property may not necessarily be satisfied in the non-integral case.
\end{rem}

In \cite{LLVWZ24}, Li, Liao, Velani, Wang and Zorin provided an example of the set \[W_2^*(t):=\{\bx\in[0,1)^2:|T_2^nx_1|<e^{-nt}\text{ and }|T_3^nx_2|<e^{-n^2}\text{ for i.m.\,$n$}\}\]
that cannot be derived from the mass transference principle from `rectangles to rectangles' by Wang and Wu \cite{WW21}, which is part of the starting point of their paper. A consequence of their main result implies that
\begin{equation}\label{eq:hdimW}
	\hdim W_2^*(t)=\min\bigg(1,\frac{\log 2+\log 3}{\log 2+t}\bigg),
\end{equation}
where $\hdim$ stands for the Hausdorff dimension. As noted in \cite[Remark 14]{LLVWZ24}, their result does not imply the Hausdorff measure of $W_2^*(t)$. Further improvements were given by the author in \cite{He24a}, where we showed that
\[\hm^s\big(W_2^*(t)\big)=\hm^s([0,1]^2)=\infty\quad\text{with $s=\hdim W_2^*(t)$.}\]
A direct consequence of Theorem \ref{t:rectangle} allows us to provide a rather complete Hausdorff measure theoretic description of $W_2^*(t)$.
\begin{thm}\label{t:example}
	Let $f$ be a dimension function such that $f\preceq 2$ and either $1\preceq f$ or $f\preceq 1$. Then, the following statements hold.
	\begin{enumerate}[(1)]
		\item Suppose that $t>\log 3$. If $1\preceq f$, then
		\[\hm^f\big(W_2^*(t)\big)=0.\]
		Conversely, if $f\prec 1$, then
		\[\hm^f\big(W_2^*(t)\big)=\begin{cases}
			0&\text{ if $\sum_{n=1}^{\infty}f(2^{-n}e^{-nt})6^n<\infty$},\\
			\hm^f([0,1]^2)&\text{ if  $\sum_{n=1}^{\infty}f(2^{-n}e^{-nt})6^n=\infty$}.
		\end{cases}\]
		\item Suppose that $t\le\log 3$. If $f\prec 1$, then
		\[\hm^f\big(W_2^*(t)\big)=\hm^f([0,1]^2)=\infty.\]
		Conversely, if $1\preceq f\preceq (1+\log 2/\log 3)$, then
		\[\begin{split}
			\hm^f\big(W_2^*(t)\big)
			=\begin{cases}
				0&\text{ if $\sum_{n=1}^{\infty}f(3^{-n}e^{-n^2})9^ne^{n^2-nt}<\infty$},\\
				\hm^f([0,1]^2)&\text{ if $\sum_{n=1}^{\infty}f(3^{-n}e^{-n^2})9^ne^{n^2-nt}=\infty$}.
			\end{cases}
		\end{split}\]
	\end{enumerate}
\end{thm}
\begin{rem}
	The condition $f\preceq (1+\log 2/\log 3)$ in item (2) is less restrictive, given that $\hdim W_2^*(t)\le 1$ (see \eqref{eq:hdimW}). Let $f$ be a dimension function defined by $f(r)=r^s/\sqrt{|\log r|}$, where $s=\hdim W_2^*(t)$. Then, $s\prec f$. By Theorem \ref{t:example}, it is easily verified that
	\[\hm^f\big(W_2^*(t)\big)=\hm^f([0,1]^2)=\infty,\]
	which cannot be covered by any of the previous results.
\end{rem}

For the set $W_d^\times(\psi,\bm h)$, $\beta_i$ is allowed to be a non-integer.
\begin{thm}\label{t:multiplicative}
Let $1<\beta_1\le\cdots\le \beta_d$.
\begin{enumerate}[(1)]
	\item  Assume that $d=1$ and $g$ is a dimension function such that $g\preceq 1$. Write $\beta=\beta_1$. Then,
	\[\hm^g\big(W_1(\psi,h)\big)=\begin{cases}
		0&\text{ if $\sum_{n=1}^{\infty}\beta^{n}g\big(\beta^{-n}\psi(n)\big)<\infty$},\\
		\hm^g([0,1])&\text{ if $\sum_{n=1}^{\infty}\beta^{n}g\big(\beta^{-n}\psi(n)\big)=\infty$}.
	\end{cases}\]
	\item  Assume that $d\ge 2$ and $f$ is a dimension function such that $(d-1)\prec f\preceq s$ for some $s\in(d-1,d)$. Then,
\end{enumerate}
	\[\hm^f\big(W_d^\times(\psi,\bm h)\big)=\begin{cases}
		0&\text{ if $\sum_{n=1}^{\infty}\beta_d^{dn}\psi(n)^{-d+1}f\big(\beta_d^{-n}\psi(n)\big)<\infty$},\\
		\hm^f([0,1]^d)&\text{ if $\sum_{n=1}^{\infty}\beta_d^{dn}\psi(n)^{-d+1}f\big(\beta_d^{-n}\psi(n)\big)=\infty$}.
	\end{cases}\]
\end{thm}
\begin{rem}
	The main ingredient of Theorem \ref{t:multiplicative} is the dichotomy law for the Hausdorff measure of $W_1(\psi,h)$, which has not been proven in the literature. For $d\ge 2$, the convergence part can be proved by combining the methods from \cite{HS18} and \cite{LLVZ23}. For the divergence part, the key is to determine the Hausdorff measure of the set $W_1(\psi,h)$ and then apply the Slicing Lemma (see Lemma \ref{l:slicing}) to complete the proof. It was shown in \cite{Phi67} that the Lebesgue measure of $W_1(\psi,h)$ satisfies a zero-one law, depending on whether the series $\sum_{n=1}^\infty \psi(n)$ converges or not. For integral $\beta>1$, the mass transfer principle by Beresnevich and Velani \cite{BV06} can be used to determine the Hausdorff measure of $W_1(\psi,h) $. However, this principle cannot be directly applied to non-integeral $\beta $, because the corresponding $\beta$-transformation is generally not a Markov map.
\end{rem}

	The structure of the paper is as follows. In the next section, we recall several notions and properties of Hausdorff measure and content in Section \ref{ss:Hausdorff}, as well as those of $\beta$-transformation in Section~\ref{ss:beta}. The proof of the convergence part of Theorem \ref{t:rectangle} is established in a separate section in Section \ref{s:converrec} without assuming each $\beta_i$ to be an integer. Based on crucial geometric observations (see Section \ref{ss:observation}), the proof of the complementary divergence part is presented in Section \ref{s:diverrec}. In Section \ref{s:example}, we prove Theorem \ref{t:example}. In Section \ref{s:onedim}, we establish the dichotomy law for the Hausdorff measure of $W_1(\psi,h)$.  Finally, the dichotomy law is applied in the last section to complete the proof of Theorem \ref{t:multiplicative} (2).

	\section{Preliminaries}\label{s:preliminaries}
	\subsection{Hausdorff measure and content}\label{ss:Hausdorff}
	Let $f$ be a dimension function. For a set $A\subset \R^d$ and $\eta>0$, let
	\[\mathcal H_\eta^f(A)=\inf\bigg\{\sum_{i}f(|B_i|):A\subset \bigcup_{i\ge 1}B_i, \text{ where $B_i$ are balls with $|B_i|\le \eta$}\bigg\},\]
	where $|E|$ denotes the diameter of $E$.
	The {\em Hausdorff $f$-measure} of $A$ is defined as
	\[\hm^f(A):=\lim_{\eta\to 0}\mathcal H_\eta^f(A).\]
	When $\eta=\infty$, $\hc^f(A)$ is referred to as {\em Hausdorff $f$-content} of $A$.


	In \cite{He24a}, the author introduced a class of $G_\delta$-subsets satisfying a certain Hausdorff $f$-content condition, and subsequently showed that any $G_\delta$-subset in this class has full Hausdorff $f$-measure and large intersection property.
	\begin{defn}[{\cite[Definition 2.3]{He24a}}]\label{d:LIP}
		Let $f$ be a dimension function. We define $\scg^f([0,1]^d)$ to be the class of $G_\delta$-subsets $A$ of $[0,1]^d$ such that there exists a constant $c>0$,
		\[\hc^f(A\cap B)>c\hc^f(B)\quad\text{holds for any ball $B\subset [0,1]^d$}.\]
	\end{defn}
	The class $\scg^f([0,1]^d)$ has the following properties.
	\begin{thm}[{\cite[Theorem 2.4]{He24a}}]\label{t:LIPfull}
		Let $f$ be a dimension function such that $f\preceq d$. The class $\scg^f([0,1]^d)$ is closed under countable intersection. Moreover, for any $A\in\scg^f([0,1]^d)$, we have
		\[\hm^f(A)=\hm^f([0,1]^d).\]
	\end{thm}

	In the same paper, the author gave two conditions that are much easier to verify whether a set belongs to $\scg^f([0,1]^d)$.

	\begin{thm}[{\cite[Theorem 2.5 and Corollary 2.6]{He24a}}]\label{t:weaken}
		Let $f$ be a dimension function such that $f\preceq d$. Assume that $\{B_n\}$ is a sequence of balls in $[0,1]^d$ with radii tending to 0, and that $\lm^d(\limsup B_n)=1$. Let $\{E_n\}$ be a sequence of open sets (not necessarily contained in $B_n$). Then,
		\[\limsup_{n\to\infty} E_n\in\scg^f([0,1]^d)\]
		if one of the following holds:
		\begin{enumerate}[(1)]
			\item $E_n\subset B_n$ for all $n\ge 1$. Moreover, there exists a constant $c>0$ such that for any $n\ge 1$,
			\[\hc^f (E_n)>c\lm^d(B_n).\]
			\item  There exists a constant $c>0$ such that for any $B_k$,
			\[\limsup_{n\to\infty} \hc^f (E_n\cap  B_k)>c\lm^d(B_k).\]
		\end{enumerate}
	\end{thm}
	With these results now at our disposal, the main ideas to prove the divergence parts of Theorems \ref{t:rectangle} and \ref{t:multiplicative} are to verify certain sets under consideration satisfying some Hausdoff $f$-content bound. For this purpose, the following mass distribution principle will be crucial.
	\begin{prop}[Mass distribution principle {\cite[Lemma 1.2.8]{BiPe17}}]\label{p:MDP}
		Let $ A $ be a Borel subset of $ \R^d $. If $ A $ supports a strictly positive Borel measure $ \mu $ that satisfies
		\[\mu(B)\le cf(|B|),\]
		for some constant $ 0<c<\infty $ and for every ball $B$, then $ \hc^f(A)\ge\mu(A)/c $.
	\end{prop}
	\subsection{$\beta$-transformation}\label{ss:beta}
	We start with a brief discussion that sums up various fundamental properties of $ \beta $-transformation.

	For $ \beta>1 $, let $ T_\beta $ be the $\beta$-transformation on $ [0,1) $.
	For any $ n\ge 1 $ and $ x\in[0,1) $, define $ \epsilon_n(x,\beta)=\lfloor \beta T_\beta^{n-1}x\rfloor $, where $\lfloor x\rfloor$ denotes the integer part of $x$. Then, we can write
	\[x=\frac{\epsilon_1(x,\beta)}{\beta}+\frac{\epsilon_2(x,\beta)}{\beta^2}+\cdots+\frac{\epsilon_n(x,\beta)}{\beta^n}+\cdots,\]
	and we call the sequence
	\[\epsilon(x,\beta):=(\epsilon_1(x,\beta),\epsilon_2(x,\beta),\dots)\]
	the $\beta$-expansion of $ x $. By the definition of $ T_\beta $, it is clear that, for $ n\ge 1 $, $ \epsilon_n(x,\beta) $ belongs to the alphabet $ \{0,1,\dots,\lceil\beta-1\rceil\} $, where $ \lceil x\rceil $ denotes the smallest integer greater than or equal to $ x $. When $ \beta $ is not an integer, then not all sequences of $ \{0,1,\dots,\lceil\beta-1\rceil\}^\N $ are the $ \beta $-expansion of some $ x\in[0,1) $. This leads to the notion of $\beta$-admissible sequence.

	\begin{defn}
		A finite or an infinite sequence $ (\epsilon_1,\epsilon_2,\dots)\in\{0,1,\dots,\lceil\beta-1\rceil\}^\N $ is said to be $\beta$-admissible if there exists an $ x\in[0,1) $ such that the $\beta$-expansion of $ x $ begins with $ (\epsilon_1,\epsilon_2,\dots) $.
	\end{defn}

	Denote by $ \Sigma_\beta^n $ the collection of all admissible sequences of length $ n $. The following result of R\'enyi~\cite{Renyi1957} implies that the cardinality of $ \Sigma_\beta^n $ is comparable to $ \beta^n $.
	\begin{lem}[{\cite[(4.9) and (4.10)]{Renyi1957}}]\label{l:renyi}
		Let $ \beta>1 $. For any $ n\ge 1 $,
		\[\beta^n\le \# \Sigma_\beta^n\le \frac{\beta^{n+1}}{\beta-1},\]
		where $ \# $ denotes the cardinality of a finite set.
	\end{lem}

	\begin{defn}
		For any $ \be_n:=(\epsilon_1,\dots,\epsilon_n)\in\Sigma_\beta^n $, we call
		\[I_{n,\beta}(\be_n):=\{x\in[0,1):\epsilon_j(x,\beta)=\epsilon_k,1\le k\le n\}\]
		an $ n $th level cylinder.
	\end{defn}
	Clearly, for any $ \be_n\in\Sigma_\beta^n $, $ T_\beta^n|_{I_{n,\beta}(\be_n)} $ is linear with slope $ \beta^n $, and it maps the cylinder $ I_{n,\beta}(\be_n) $ into $ [0,1) $. If $ \beta $ is not an integer, then the dynamical system $ (T_\beta, [0,1)) $ is not a full shift, and so $ T_\beta^n|_{I_{n,\beta}(\be_n)} $ is not necessary onto. In other words, the length of $ I_{n,\beta}(\be_n) $ may strictly less than $ \beta^{-n} $, which makes describing the dynamical properties of $ T_\beta $ more challenging. To get around this barrier, we need the following notion.
	\begin{defn}
		A cylinder $ I_{n,\beta}(\be_n) $ or a sequence $ \be_n\in\Sigma_\beta^n $ is called $\beta$-full if it has maximal length, that is, if
		\[|I_{n,\beta}(\be_n)|=\beta^{-n}.\]
	\end{defn}
	When there is no risk of ambiguity, we will write full instead of $\beta$-full. The importance of full sequences is based on the fact that the concatenation of any two full sequences is still full.
	\begin{prop}[{\cite[Lemma 3.2]{FaWa2012}}]\label{p:concatenation}
		An $ n $th level cylinder $ I_{n,\beta}(\be_n) $ is full if and only if, for any $\beta$-admissible sequence $ \be'_m\in\Sigma_\beta^m $ with $ m\ge 1 $, the concatenation $ \be_n\be_m' $ is still $ \beta $-admissible. Moreover,
		\[|I_{n+m,\beta}(\be_n\be_m')|=|I_{n,\beta}(\be_n)|\cdot|I_{m,\beta}(\be_m')|.\]
		So, for any two full cylinders $ I_{n,\beta}(\be_n), I_{m,\beta}(\be_m') $, the cylinder $ I_{n+m,\beta}(\be_n\be_m') $ is also full.
	\end{prop}
	For an interval $ I\subset [0,1) $, let $ \Lambda_{\beta}^n(I) $ denote the set of full sequences $ \be_n\in\Sigma_\beta^n $ with $ I_{n,\beta}(\be_n)\subset I $. In particular, if $ I=[0,1) $, then we simply write $ \Lambda_{\beta}^n $ instead of $ \Lambda_{\beta}^n([0,1)) $. For this case, the cardinality of $\Lambda_{\beta}^n$ can be estimated as follows:
	\begin{lem}[{\cite[Lemma 1.1.46]{Li21}}]\label{l:Li}
		Let $\beta>1$ and $n\in \N$.
		\begin{enumerate}[(1)]
			\item If $\beta\in\N$, then
			\[\#\Lambda_{\beta}^n=\beta^n.\]
			\item If $\beta>2$, then
			\[\#\Lambda_{\beta}^n>\frac{\beta-2}{\beta-1}\beta^n.\]
			\item If $1<\beta<2$, then
			\[\#\Lambda_{\beta}^n>\bigg(\prod_{i=1}^{\infty}(1-\beta^{-i})\bigg)\beta^n.\]
		\end{enumerate}
	\end{lem}
	Since the concatenation of any two full sequences remains full, we can derive the following corollary.
	\begin{cor}\label{c:inside I}
		For any full sequence $\be_m\in\Lambda_\beta^m$, there is a constant $c_\beta>0$ such that for any $n\ge m$
		\[\#\Lambda_\beta^n\big(I_{m,\beta}(\be_m)\big)=\#\Lambda_\beta^{n-m}\ge c_\beta\beta^n|I_{m,\beta}(\be_m)|.\]
	\end{cor}
	The following lemma will be used in establishing the dichotomy law for the Hausdorff measure of $W_1(\psi,h)$.
	\begin{lem}\label{l:union=[0,1]}
		We have
		\[\bigcap_{N=1}^\infty\bigcup_{n=N}^\infty\bigcup_{\be_n\in\Lambda_{\beta}^n}I_{n,\beta}(\be_n)=[0,1).\]
		In particular, the $\limsup$ set defined by all full cylinders has full Lebesgue measure.
	\end{lem}
	\begin{proof}
		It is sufficient to show that for all $N\ge 1$,
		\begin{equation}\label{eq:union=[0,1]}
			\bigcup_{n=N}^\infty\bigcup_{\be_n\in\Lambda_{\beta}^n}I_{n,\beta}(\be_n)=[0,1).
		\end{equation}
		Although the union of $N$th level full cylinders
		\[\bigcup_{\be_N\in\Lambda_{\beta}^N}I_N(\be_N)\]
		usually fails to cover $[0,1)$, those uncovered interval can be covered by higher level cylinders. To see this, choose an $N$th level non-full sequence $\be_N'\in\Sigma_\beta^N\setminus\Lambda_\beta^N$. Let $n>N$ be the unique integer satisfying
		\[\beta^{-n}\le |I_N(\be_N')|<\beta^{-n+1}.\]
		Then, the sequence $\be_n:=\be_N0^{n-N}$ is full and $I_{n,\beta}(\be_n)\subset I_{N,\beta}(\be_N)$, where $0^{n-N}$ denotes a sequence of length $n-N$ consisting solely of $0$.
		Therefore, at least $\beta^{-1}$ of $I_N(\be_N')$ can be covered by $n$th level full cylinders. Continue this procedure, we arrive at \eqref{eq:union=[0,1]}.
	\end{proof}
	Part of the proof of the following lemma can be found in \cite{Wang18}. We include the proof for completeness.
	\begin{lem}\label{l:length}
		Let $h$ be a  Lipschitz function with Lipschitz constant $L\ge0$. Let $0<r<1$. For any $n$ with $ L<\beta^n $ and any sequence $\be_n\in\Sigma_\beta^n$, the set
		\[A:=\{x\in I_{n,\beta}(\be_n):|T_\beta^nx-h(x)|<r\}\]
		is contained in a ball of radius $2r\beta^{-n}$. Moreover, if $\be_n$ is full, then $A$ contains a ball of radius $r\beta^{-n}/2$.
	\end{lem}
	\begin{proof}
		Choose two points $ x,y\in A $. By using the triangle inequality, for any $n$ with $ L<\beta^n $, we have
		\[\begin{split}
			2r&\ge\big|\big(T_\beta^nx-h(x)\big)-\big(T_\beta^ny-h(y)\big)\big|\ge |T_\beta^nx-T_\beta^ny|-|h(x)-h(y)|\\
			&\ge\beta^n|x-y|-L|x-y|=(\beta^n-L)|x-y|.
		\end{split}\]
		This implies that $ A $ is contained in a ball of radius $ 2r/|\beta|^n $.

		Now, suppose that $\be_n$ is full. Let $x^*$ and $y^*$ denote the left and right endpoints of $I_{n,\beta}(\be_n)$, respectively. Then, $T_\beta^n$ maps $I_{n,\beta}(\be_n)$ onto $[0,1)$, and
		\[T_\beta^nx^*=0\qaq T_\beta^ny^*=1.\]
		Since
		\[T_\beta^nx^*-h(x^*)=0-h(x^*)\le 0\qaq T_\beta^ny^*-h(y^*)=1-h(y^*)\ge 0,\]
		there exists a point $z\in I_{n,\beta}(\be_n)$ such that
		\[T_\beta^nz-h(z)=0,\quad\text{or equivalently,}\quad T_\beta^nz=h(z).\]
		Let $x\in I_{n,\beta}(\be_n)$ with $|x-z|<r\beta^{-n}/2$. It follows from  $T_\beta^nz=h(z)$ that
		\[\begin{split}
			|T_\beta^nx-h(x)|&=|T_\beta^nx-T_\beta^nz+h(z)-h(x)|\\
			&\le |T_\beta^nx-T_\beta^nz|+|h(x)-h(z)|\\
			&\le\beta^n|x-z|+L|x-z|\\&<r/2+rL\beta^{-n}/2=r.
		\end{split}\]
		Therefore, $A$ contains a ball of radius $r\beta^{-n}/2$.
	\end{proof}
	\begin{rem}\label{r:samecenter}
		If $\be_n$ is full, then the center of the big ball containing $A$ and the relatively smaller ball contained in $A$ can be made identical. Consequently, the center of these balls is independent of $r$. To see this, in proving the first point of Lemma \ref{l:length}, one can select $y$ to be $z$ with $T_\beta^nz=h(z)$, and the rest of the proof is unchanged.
	\end{rem}

\section{Proof of Theorem \ref{t:rectangle}: Convergence part}\label{s:converrec}
The proof of the convergence part of Theorem \ref{t:rectangle} does no require that each $\beta_i$ is an integer. Suppose that \[\sum_{n=1}^{\infty} s_n(\Psi,f)\prod_{i=1}^d\beta_i^n<\infty,\]
where recall that
	\[s_n(\Psi,f)=\min_{\tau\in\ca_n}\biggl\{f(\tau)\prod_{i\in\ck_{n,1}(\tau)}\frac{\beta_i^{-n}}{\tau}\prod_{i\in\ck_{n,2}(\tau)}\frac{\beta_i^{-n}\psi_i(n)}{\tau}\biggr\},\]
and, in turn $\ca_n=\{\beta_1^{-n},\dots,\beta_d^{-n},\beta_1^{-n}\psi_(n),\dots,\beta_d^{-n}\psi_d(n)\}$,
\[\ck_{n,1}(\tau):=\{i:\beta_i^{-n}\le \tau\}\quad\text{and}\quad\ck_{n,2}(\tau):=\{i:\beta_i^{-n}\psi_i(n)\ge \tau\}.\]

By Lemma \ref{l:length}, for each $\beta_i$ and sufficiently large $n$, we have for any $\be_n^i\in\Sigma_{\beta_i}^n$,
\[\{x_i\in I_{n,\beta_i}(\be_n^i):|T_{\beta_i}^nx_i-h_i(x_i)|<\psi_i(n)\}\]
is contained in a ball of radius $2\beta_i^{-n}\psi_i(n)$. Each $\be_n^i\in\Sigma_{\beta_i}^n$ corresponds to such a ball of the same radius. Denote by $\ce_{n,i}$ the collection of centers of these balls. By Lemma \ref{l:renyi},
\[\#\ce_{n,i}\asymp\beta_i^n.\]
On letting $\ce_n=\ce_{n,1}\times\cdots\times \ce_{n,d}$, we have
\begin{equation}\label{eq:subset}
	W_d(\Psi,\bm h)\subset\bigcap_{N=1}^\infty\bigcup_{n=N}^\infty\bigcup_{\bz\in\mathcal E_n}\prod_{i=1}^{d}B\big(z_i,2\beta_i^{-n}\psi_i(n)\big),
\end{equation}
where $\bz=(z_1,\dots,z_d)$. The task is to find an optimal cover for the collection of hyperrectangles
\[E_n:=\bigcup_{\bz\in\mathcal E_n}\prod_{i=1}^{d}B\big(z_i,2\beta_i^{-n}\psi_i(n)\big)=\prod_{i=1}^d\bigcup_{z_i\in\mathcal E_{n,i}}B\big(z_i,2\beta_i^{-n}\psi_i(n)\big).\]
For the Cartesian product $A=\prod_{i=1}^{d}A_i$, call $A_i$ the part of $A$ in the $i$th direction.

 Fix $\tau\in\ca_n$. We have the following observations:
\begin{enumerate}[(1)]
	\item in the $i$th direction with $i\in \ck_{n,1}(\tau)$, the unit interval $[0,1)$ can be covered by
	\[\asymp\tau^{-1}\]
	balls of radius $\tau$. Such collection of balls will cover $\bigcup_{z_i\in\mathcal E_{n,i}}B\big(z_i,2\beta_i^{-n}\psi_i(n)\big)$ as well.
	\item in the $i$th direction with $i\in\ck_{n,2}(\tau)$, since $\tau<\beta_i^{-n}\psi_i(n)$ one needs
	\[\asymp\#\mathcal E_{n,i}\cdot \bigg(\frac{\beta_i^{-n}\psi_i(n)}{\tau}\bigg)\asymp \frac{\psi_i(n)}{\tau}\]
	balls of radius $\tau$ to cover $\bigcup_{z_i\in\mathcal E_{n,i}}B\big(z_i,2\beta_i^{-n}\psi_i(n)\big)$.
	\item in the $i$th direction with $i\notin\ck_{n,1}\cup\ck_{n,2}(\tau)$, since $\beta_i^{-n}\psi_i(n)\le \tau\le \beta_i^{-n}$, one needs
	\[\asymp\beta_i^n\]
	balls of radius $\tau$ to cover $\bigcup_{z_i\in\mathcal E_{n,i}}B\big(z_i,2\beta_i^{-n}\psi_i(n)\big)$.
\end{enumerate}
In summary, $E_n$ can be covered by
\[\begin{split}
	\asymp&\prod_{i\in\ck_{n,1}(\tau)}\tau^{-1}\prod_{i\in\ck_{n,2}(\tau)}\frac{\psi_i(n)}{\tau}\prod_{i\notin\ck_{n,1}\cup\ck_{n,2}(\tau)}\beta_i^n\\
	=& \prod_{i=1}^d\beta_i^n\prod_{i\in\ck_{n,1}(\tau)}\frac{\beta_i^{-n}}{\tau}\prod_{i\in\ck_{n,2}(\tau)}\frac{\beta_i^{-n}\psi_i(n)}{\tau}
\end{split}\]
balls of radius $\tau$. By taking the minimum over $\tau\in\ca_n$, the $f$-volume cover of $E_n$ is majorized by
\[\begin{split}
	\asymp\min_{\tau\in\ca_n}\biggl\{f(\tau)\prod_{i=1}^d\beta_i^n\prod_{i\in\ck_{n,1}(\tau)}\frac{\beta_i^{-n}}{\tau}\prod_{i\in\ck_{n,2}(\tau)}\frac{\beta_i^{-n}\psi_i(n)}{\tau}\biggr\}=s_n(\Psi,f)\prod_{i=1}^d\beta_i^n.
\end{split}\]
By the definition of Hausdorff $f$-measure and \eqref{eq:subset}, we have
\[\hm^f\big(W_d(\Psi,\bm h)\big)\le\limsup_{N\to\infty}\sum_{n=N}^{\infty}s_n(\Psi,f)\prod_{i=1}^d\beta_i^n=0,\]
since $\sum_{n=1}^{\infty}s_n(\Psi,f)\prod_{i=1}^d\beta_i^n<\infty$.

\section{Proof of Theorem \ref{t:rectangle}: Divergence part}\label{s:diverrec}
In this section, we assume that each $\beta_i$ is an integer and that
\[\sum_{n=1}^\infty s_n(\Psi,f)\prod_{i=1}^d\beta_i^n=\infty.\]
In this case, we have $\Lambda_{\beta_i}^n=\Sigma_{\beta_i}^n$ for all $1\le i\le d$ and $n\ge 1$. Moreover,
\[\#\Lambda_{\beta_i}^n=\beta_i^n.\]
For each $\be_n^i\in\Lambda_{\beta_i}^n$, by Lemma \ref{l:length}, the set
\[\{x_i\in I_{n,\beta_i}(\be_n^i):|T_{\beta_i}^nx_i-h_i(x_i)|<\psi_i(n)\}\]
contains a ball of radius $\beta_i^{-n}\psi_i(n)/2$. Each $\be_n^i\in\Sigma_{\beta_i}^n$ corresponds to such a ball of the same radius. With slightly abuse of notation, still denote by $\ce_{n,i}$ the collection of centers of these balls. Write $\ce_n=\ce_{n,1}\times\cdots\times \ce_{n,d}$. It follows that
\begin{equation}\label{eq:wd}
		\bigcap_{N=1}^\infty\bigcup_{n=N}^\infty\bigcup_{\bz\in\mathcal E_n}\prod_{i=1}^{d}B\big(z_i,\beta_i^{-n}\psi_i(n)/2\big)\subset W_d(\Psi,\bm h)\subset\bigcap_{N=1}^\infty\bigcup_{n=N}^\infty\bigcup_{\bz\in\mathcal E_n}\prod_{i=1}^{d}B\big(z_i,2\beta_i^{-n}\psi_i(n)\big).
\end{equation}
The following Lebegue measure theoretic law for $W_d(\Psi,\bm h)$ is a variant version due to Wang and Kleinbock \cite[Theorem 2.6]{KW23}. Although their original statement focus on the case $\bm h(\bx)\equiv\ba\in[0,1)^d$, the general statement below actually holds with the help of \eqref{eq:wd}.

\begin{thm}\label{t:rectangleleb}
	Suppose that each $\beta_i$ is an integer. Then,
	\[\lm^d\big(W_d(\Psi,\bm h)\big)=\begin{cases}
		0&\text{ if $\sum_{n=1}^{\infty}\prod_{i=1}^{d}\psi_i(n)<\infty$},\\
		1&\text{ if $\sum_{n=1}^{\infty}\prod_{i=1}^{d}\psi_i(n)=\infty$}.
	\end{cases}\]
\end{thm}
%
\subsection{A reduction argument}\label{ss:reduction}
Since $f\preceq d$, there are two situations to consider:
\begin{enumerate}[(1)]
	\item $\lim_{r\to0^+}f(r)/r^d=c<\infty$,
	\item $\lim_{r\to0^+}f(r)/r^d=\infty$.
\end{enumerate}

If we are in situation (1), then $\hm^f=c\lm^d$ and
\[\text{$f(r)\asymp r^d$ for all small $r>0$}.\]
Hence, for sufficiently large $n$,
\begin{align}
		s_n(\Psi,f)\prod_{i=1}^d\beta_i^n
	=&\min_{\tau\in\ca_n}\biggl\{f(\tau)\prod_{i=1}^d\beta_i^n\prod_{i\in\ck_{n,1}(\tau)}\frac{\beta_i^{-n}}{\tau}\prod_{i\in\ck_{n,2}(\tau)}\frac{\beta_i^{-n}\psi_i(n)}{\tau}\biggr\}\notag\\
	\asymp& \min_{\tau\in\ca_n}\biggl\{\tau^d\prod_{i=1}^d\beta_i^n\prod_{i\in\ck_{n,1}(\tau)}\frac{\beta_i^{-n}}{\tau}\prod_{i\in\ck_{n,2}(\tau)}\frac{\beta_i^{-n}\psi_i(n)}{\tau}\biggr\}\notag\\
	=& \min_{\tau\in\ca_n}\biggl\{\prod_{i\notin\ck_{n,1}(\tau)\cup\ck_{n,2}(\tau)}\tau\beta_i^n\prod_{i\in\ck_{n,2}(\tau)}\psi_i(n)\biggr\}.\label{eq:sn}
\end{align}
Since for $i\notin \ck_{n,1}(\tau)\cup\ck_{n,2}(\tau)$, $\beta_i^{-n}\psi_i(n)<\tau<\beta_i^{-n}$. It holds that
\[\psi_i(n)<\tau\beta_i^n<1.\]
Therefore,
\[\prod_{i=1}^{d}\psi_i(n)\le \prod_{i\notin\ck_{n,1}(\tau)\cup\ck_{n,2}(\tau)}\psi_i(n)\prod_{i\in\ck_{n,2}(\tau)}\psi_i(n)<\prod_{i\notin\ck_{n,1}(\tau)\cup\ck_{n,2}(\tau)}\tau\beta_i^n\prod_{i\in\ck_{n,2}(\tau)}\psi_i(n).\]
This holds for all $\tau\in \ca_n$. It follows that
\begin{equation}\label{eq:sn>>}
	s_n(\Psi,f)\prod_{i=1}^d\beta_i^n\gg \prod_{i=1}^{d}\psi_i(n).
\end{equation}

On the other hand, for $\tau_{\min}=\min_{\tau\in\ca_n}\tau$, we have
\[\ck_{n,1}(\tau_{\min})=\emptyset\qaq\ck_{n,2}(\tau_{\min})=\{1,\dots,d\}.\]
By \eqref{eq:sn},
\begin{equation}\label{eq:sn<}
	s_n(\Psi,f)\prod_{i=1}^d\beta_i^n\ll \prod_{i=1}^{d}\psi_i(n).
\end{equation}
Combining \eqref{eq:sn>>} and \eqref{eq:sn<}, we have
\[	s_n(\Psi,f)\prod_{i=1}^d\beta_i^n\asymp \prod_{i=1}^{d}\psi_i(n).\]
It follows from $\hm^f=c\lm^d$ that
\[\begin{split}
	\sum_{n=1}^\infty s_n(\Psi,f)\prod_{i=1}^d\beta_i^n<\infty&\quad\Longleftrightarrow\quad\sum_{n=1}^\infty \prod_{i=1}^d\psi_i(n)<\infty\\
	&\quad\Longrightarrow\quad\hm^f\big(W_d(\Psi,\bm h)\big)=c\lm^d\big(W_d(\Psi,\bm h)\big)=0,
\end{split}\]
and
\[\begin{split}
	\sum_{n=1}^\infty s_n(\Psi,f)\prod_{i=1}^d\beta_i^n=\infty&\quad\Longleftrightarrow\quad\sum_{n=1}^\infty \prod_{i=1}^d\psi_i(n)=\infty\\
	&\quad\Longrightarrow\quad\hm^f\big(W_d(\Psi,\bm h)\big)=c\lm^d\big(W_d(\Psi,\bm h)\big)=c=\hm^f([0,1]^d),
\end{split}\]
which finishes the proof of situation (1).

From now on, suppose that
\begin{equation}\label{eq:=infty}
	\lim_{r\to0^+}f(r)/r^d=\infty.
\end{equation}
Note that \eqref{eq:=infty} implies that $f(r)>r^d$ for all small $r>0$. Following the same lines as the proof of \eqref{eq:sn>>}, it follows from \eqref{eq:=infty} that for large $n$,
\begin{equation}\label{eq:sn>}
	s_n(\Psi,f)\prod_{i=1}^d\beta_i^n> \prod_{i=1}^{d}\psi_i(n).
\end{equation}
\subsection{Geometric interpretation of the method}\label{ss:observation}
Before giving the technical proof of the divergence part of Theorem \ref{t:rectangle}, we explain the nature of our method.

 For simplicity, suppose that $d=2$ and $\bm h(\bx)\equiv (1/2,1/2)$. Then, $W_2(\Psi,\bm h)$ can be written as the $\limsup$ set of rectangles
\[W_2(\Psi,\bm h)=\bigcap_{N=1}^\infty\bigcup_{n=N}^\infty\bigcup_{\bz\in\mathcal E_n}\prod_{i=1}^2B\big(z_i,\beta_i^{-n}\psi_i(n)\big)=:\bigcap_{N=1}^\infty\bigcup_{n=N}^\infty E_n,\]
where $\ce_n$ is given in \eqref{eq:wd}. Let $f$ be a dimension function. Based on Theorem \ref{t:weaken} (1), the main goal is to find a $\limsup$ set of balls $B_k$ and a sequence of sets $F_k$ such that
\[\lm^2\Big(\limsup_{k\to\infty} B_k\Big)=1,\quad F_k\subset B_k,\quad \hc^f(F_k)\gg\lm^2(B_k)\quad\text{and}\quad\limsup_{k\to\infty}F_k\subset W_2(\Psi,\bm h).\]

Now, let us start with the construction. For simplicity, we consider two special but crucial cases.

\noindent{\bf Case 1:} $s_n(\Psi,f)\beta_2^n\ge \beta_1^{-n}\psi_1(n)$ for $n\ge1$. Define a $2$-tuple of functions $\Phi=(\phi_1,\phi_2)$ as
\[\phi_i(n)=\begin{cases}
	s_n(\Psi,f)\beta_1^n\beta_2^n&\text{if $i=1$},\\
	1&\text{if $i=2$}.
\end{cases}\]
By Theorem \ref{t:rectangleleb}, $W_2(\Phi,\bm h)$ has full Lebesgue measure. Write
\[W_2(\Phi,\bm h)=\bigcap_{N=1}^\infty\bigcup_{n=N}^\infty\bigcup_{\bz\in\mathcal E_n}\prod_{i=1}^2B\big(z_i,\beta_i^{-n}\phi_i(n)\big)=:\bigcap_{N=1}^\infty\bigcup_{n=N}^\infty \tilde E_n.\]
For each $n\ge 1$, the length of the rectangle $\prod_{i=1}^2B(z_i,\beta_i^{-n}\phi_i(n))$ is $\beta_1^{-n}\phi_1(n)=s_n(\Psi,f)\beta_2^n$ and the width is $\beta_2^{-n}$.
Since $s_n(\Psi,f)\beta_2^n\ge \beta_1^{-n}\psi_1(n)$, we have \[\prod_{i=1}^2B\big(z_i,\beta_i^{-n}\psi_i(n)\big)\subset \prod_{i=1}^2B\big(z_i,\beta_i^{-n}\phi_i(n)\big),\]
and so
\[W_2(\Psi,\bm h)\subset W_2(\Phi,\bm h).\]
See Figure \ref{f:small} (A) and (B) for an illustration. In addition, some rectangles in $\tilde E_n$ completely overlap with one side of other rectangles, thus forming a larger rectangle (see Figure \ref{f:small} (B)). With this observation, we divide the large rectangle into smaller balls with length $\beta_1^{-n}\phi_1(n)$, which consists of the desired balls $B_k$. More precisely, there exists a collection $\tilde\ce_n$ of points such that
\[\bigcup_{\bz\in\mathcal E_n}\prod_{i=1}^2B\big(z_i,\beta_i^{-n}\phi_i(n)\big)\subset\bigcup_{\by\in\tilde\ce_n}B\big(\by,\beta_1^{-n}\phi_1(n)\big).\]
For each $\by$ with $\by\in \tilde\ce_n$, let
\[\bigcup_{\bz\in\mathcal E_n}\prod_{i=1}^2B\big(z_i,\beta_i^{-n}\psi_i(n)\big)\cap B\big(\by,\beta_1^{-n}\phi_1(n)\big).\]
This gives the desired set $F_k\subset B_k$ (see Figure \ref{f:small} (C)). The final step is to verify
\[\hc^f\bigg(\bigcup_{\bz\in\mathcal E_n}\prod_{i=1}^2B\big(z_i,\beta_i^{-n}\psi_i(n)\big)\cap B\big(\by,\beta_1^{-n}\phi_1(n)\big)\bigg)\gg \lm^2\big(B\big(\by,\beta_1^{-n}\phi_1(n)\big)\big).\]
\begin{figure}
		\centering
		\subfloat[$E_2$]
		{
			\label{fig:subs1}\includegraphics[width=0.3\textwidth]{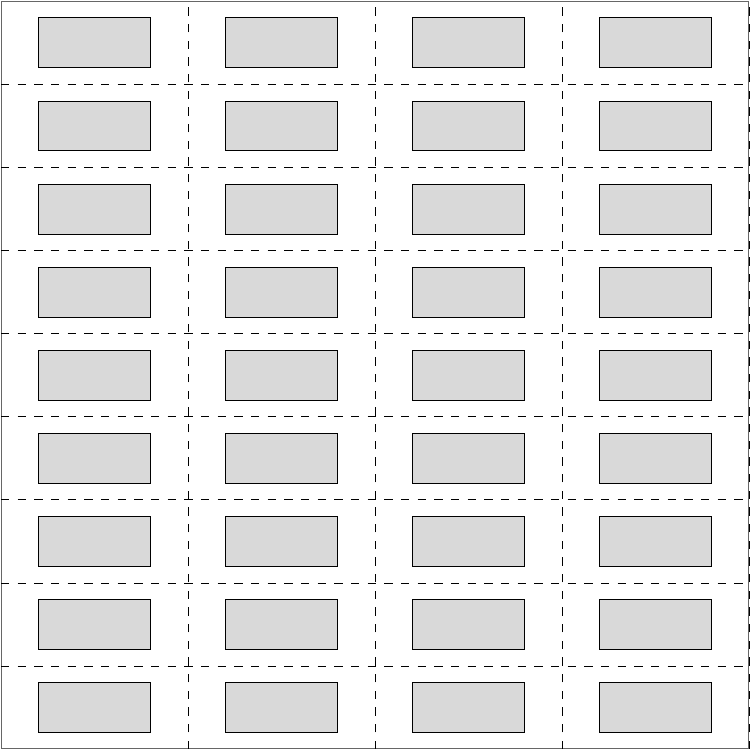}
		}
		\subfloat[$\tilde E_2$]
		{
			\label{fig:subs2}\includegraphics[width=0.3\textwidth]{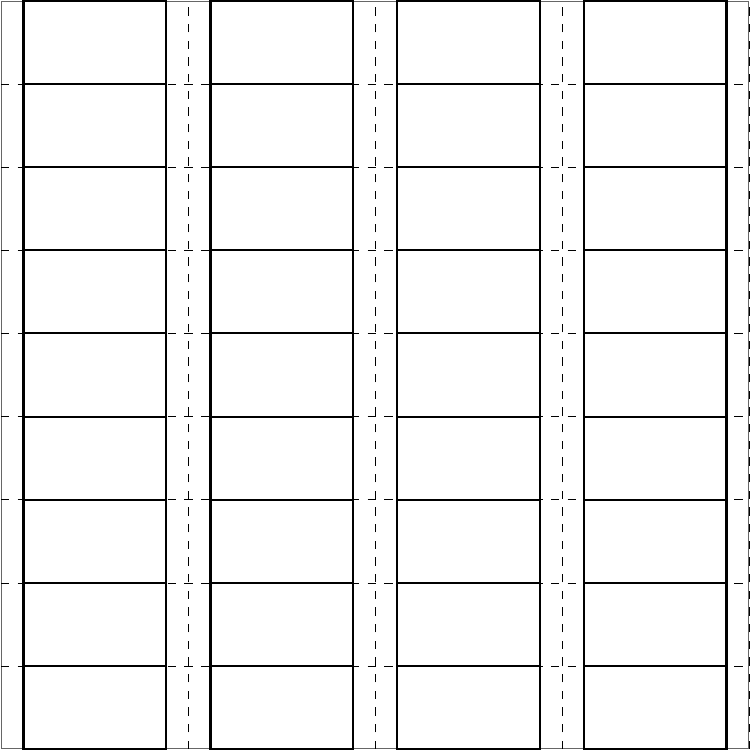}
		}
		\subfloat[desired $ F_k $ and $ B_k$]
		{
			\label{fig:subs3}\includegraphics[width=0.3\textwidth]{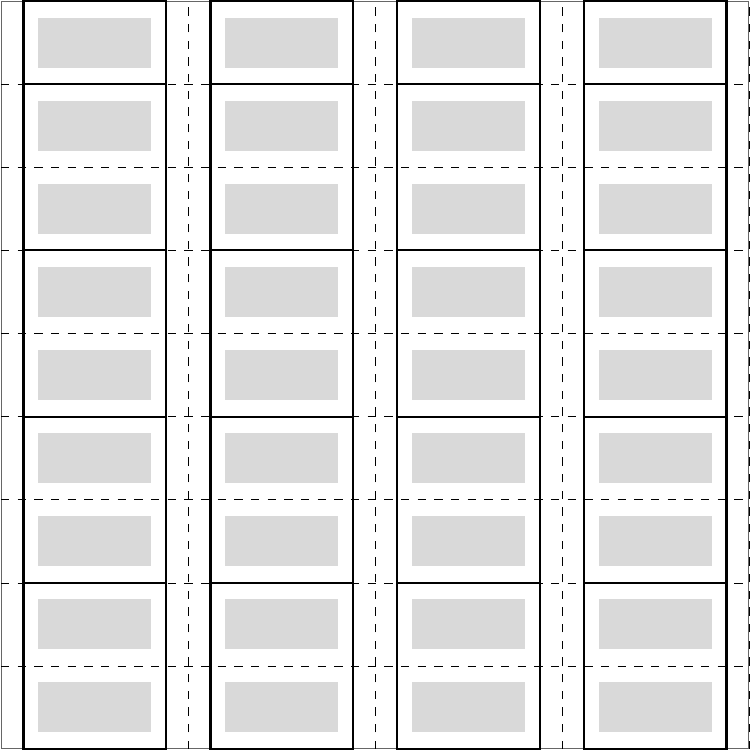}
		}
		\caption{Case 1: $s_n(\Psi,f)\beta_2^n\ge \beta_1^{-n}\psi_1(n)$.}
		\label{f:small}
	\end{figure}
	\noindent{\bf Case 2:} $s_n(\Psi,f)\beta_2^n< \beta_1^{-n}\psi_1(n)$ for $n\ge1$. In this case, one cannot define $\Phi$ in the same way as in Case 1, since we do not have $W_2(\Psi,\bm h)\subset W_2(\Phi,\bm h)$. Let $\Upsilon=(\upsilon_1,\upsilon_2)$ be defined by
	\[\upsilon_i(n)=\begin{cases}
		\psi_1(n)&\text{if $i=1$},\\
		s_n(\Psi,f)\beta_1^n\beta_2^n\psi_1(n)^{-1}&\text{if $i=2$}.
	\end{cases}\]
	By Theorem \ref{t:rectangleleb}, $W_2(\Upsilon,\bm h)$ has full Lebesgue measure. Write
	\[W_2(\Upsilon,\bm h)=\bigcap_{N=1}^\infty\bigcup_{n=N}^\infty\bigcup_{\bz\in\mathcal E_n}\prod_{i=1}^2B\big(z_i,\beta_i^{-n}\upsilon_i(n)\big)=:\bigcap_{N=1}^\infty\bigcup_{n=N}^\infty \hat E_n.\]
	As presented in Figure \ref{f:large} (A) and (B), we have
	\[W_2(\Psi,\bm h)\subset W_2(\Upsilon,\bm h).\]
	However, the geometric of the set $W_2(\Upsilon,\bm h)$ differs significantly from that in Case 1. To obtain the desired balls $B_k$, we divide the rectangle $\prod_{i=1}^2B(z_i,\beta_i^{-n}\upsilon_i(n))$ into smaller balls with length $\beta_2^{-n}\upsilon_2(n)=s_n(\Psi,f)\beta_1^n\psi_1(n)^{-1}$. More precisely, for each $\bz\in\ce_n$, there exists a collection $\hat\ce_n(\bz)$ of points such that
	\[\prod_{i=1}^2B\big(z_i,\beta_i^{-n}\upsilon_i(n)\big)\subset\bigcup_{\by\in\hat\ce_n(\bz)}B\big(\by,\beta_2^{-n}\upsilon_2(n)\big).\]
	For each $\by\in \hat\ce_n(\bz)$ with $\bz\in\ce_n$, let
	\[\prod_{i=1}^2B\big(z_i,\beta_i^{-n}\psi_i(n)\big)\cap B\big(\by,\beta_2^{-n}\upsilon_2(n)\big).\]
	This gives the desired $F_k\subset B_k$ (see Figure \ref{f:large} (C)). The final step is to verify
	\[\hm^f\bigg(\prod_{i=1}^2B\big(z_i,\beta_i^{-n}\psi_i(n)\big)\cap B\big(\by,\beta_2^{-n}\upsilon_2(n)\big)\bigg)\gg \lm^2\big(B\big(\by,\beta_2^{-n}\upsilon_2(n)\big)\big).\]
	\begin{figure}
		\centering
		\subfloat[$E_2$]
		{
			\label{fig:subl1}\includegraphics[width=0.3\textwidth]{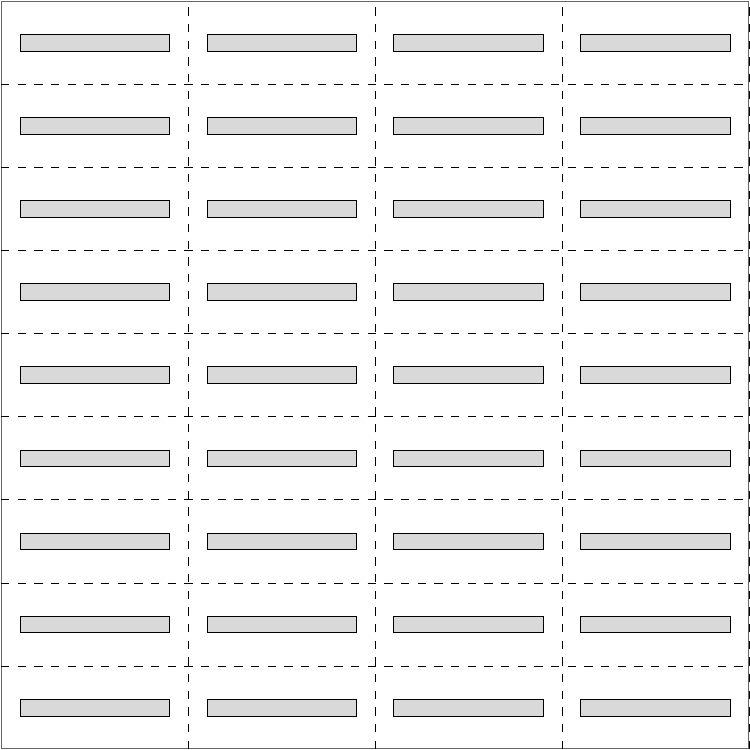}
		}
		\subfloat[$\hat E_2$]
		{
			\label{fig:subl2}\includegraphics[width=0.3\textwidth]{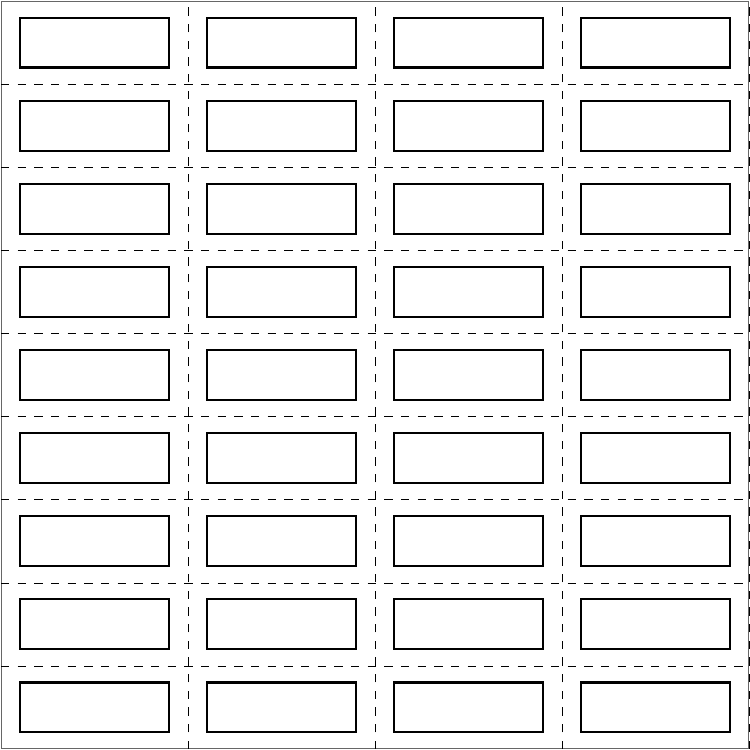}
		}
		\subfloat[desired $F_k$ and $B_k$]
		{
			\label{fig:subl3}\includegraphics[width=0.3\textwidth]{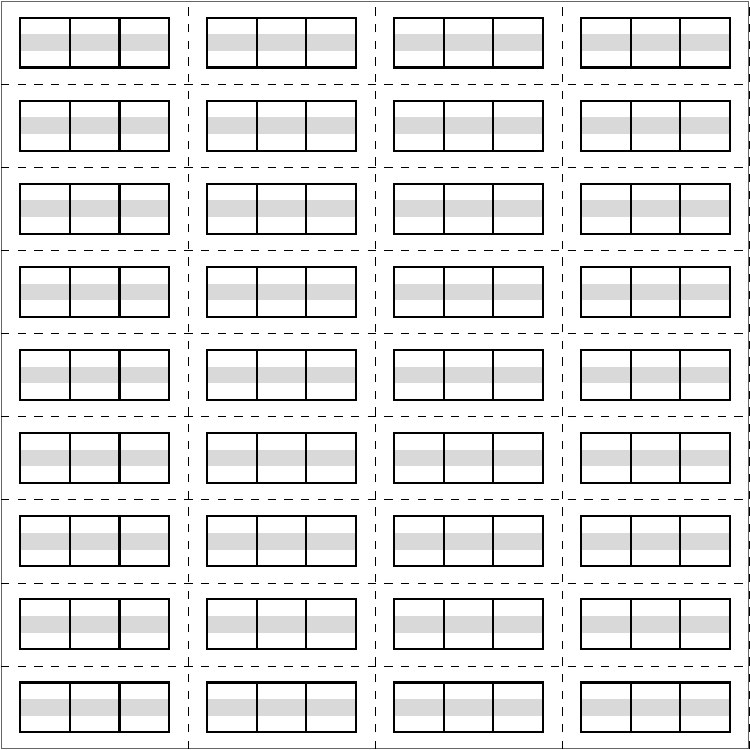}
		}
		\caption{Case 2: $s_n(\Psi,f)\beta_2^n< \beta_1^{-n}\psi_1(n)$.}
		\label{f:large}
	\end{figure}

	These two cases indeed exhaust all possibilities in $d=2$. In higher dimensions, the situation is much more complicated and requires a more in-depth analysis. Nonetheless, the geometric insights above provide valuable inspiration for extending the approach.

\subsection{Proof of Theorem \ref{t:rectangle}: Divergence case}\label{ss:completediv}
In what follows, suppose that \eqref{eq:=infty} holds, i.e.
\[\lim_{r\to0^+}f(r)/r^d=\infty.\]

Recall that $1<\beta_1\le\cdots\le\beta_d$. Let $0=k_0\le k_1<k_2<\cdots<k_s=d$ be the integers such that for $0\le j\le s-1$,
\begin{equation}
	\beta_{k_j}<\beta_{k_j+1}=\cdots=\beta_{k_{j+1}}<\beta_{k_{j+1}+1}.
\end{equation}
By pigeonhole principle, for any $0\le j\le s-1$ there is a permutation $i_{k_j+1},\dots, i_{k_j}$ of $k_{j}+1,\dots,k_{j+1}$ such that
\[\sum_{n\in P}s_n(\Psi,f)\prod_{i=1}^d\beta_i^n=\infty,\]
where
\begin{align}
	P=&\{n\in\N:\psi_{i_{k_j+1}}(n)\ge\cdots\ge\psi_{i_{k_{j+1}}}(n)\text{ for $0\le j\le s-1$}\}\notag\\
	&\cap \bigg\{n\in\N:n^{-2}\le s_n(\Psi,f)\prod_{i=1}^d\beta_i^n\le 1\bigg\}.\label{eq:sufficientlarge}
\end{align}
Note the permutation does not change $1<\beta_1\le \beta_2\le\cdots\le \beta_d$. Without loss of generalisity, we assume that for any $0\le j\le s-1$, the permutation $i_{k_j+1},\dots, i_{k_j}$ is exactly $k_{j}+1,\dots,k_{j+1}$. In other words, for $0\le j\le s-1$,
\begin{equation}\label{eq:psiincrease}
	\psi_{k_j+1}(n)\ge\cdots\ge\psi_{k_{j+1}}(n).
\end{equation}
It is worth mentioning that for $0\le j\le s-1$,
\[\text{$\psi_{k_j}(n)\ge \psi_{k_j+1}(n)$ does not necessarily hold.}\]

To simplify the notation, in what follows we write $s_n$ for $s_n(\Psi,f)$.
For any $n\ge \N$, let $0\le m=m(n)\le d-1$ be the smallest integer such that
\begin{equation}\label{eq:defmkj}
	\bigg(s_n\prod_{i=1}^{m}\psi_i(n)^{-1}\prod_{i=1}^{d}\beta_i^n\bigg)^{\frac{1}{k_j-m}}\ge\psi_{m+1}(n),
\end{equation}
where $k_j=k_j(n)$ is chosen so that $k_{j-1}\le m<k_j$ with $1\le j\le s$. If $m=0$, the product $\prod_{i=1}^{m}\psi_i(n)^{-1}$ is regarded as $1$.
Such $m$ exists for all large enough $n$, since for $m=d-1$ by \eqref{eq:sn>},
\[\bigg(s_n\prod_{i=1}^{m}\psi_i(n)^{-1}\prod_{i=1}^{d}\beta_i^n\bigg)^{\frac{1}{k_s-m}}=s_n\prod_{i=1}^{d-1}\psi_i(n)^{-1}\prod_{i=1}^{d}\beta_i^n\ge \psi_d(n).\]
Here and below, the dependence of $m$ and $k_j$ on $n$ will be implicitly understood and not explicitly stated.

\begin{lem}\label{l:increase}
	We have
	\begin{equation*}
		\begin{split}
			\bigg(s_n\prod_{i=1}^{m}\psi_i(n)^{-1}\prod_{i=1}^{d}\beta_i^n\bigg)^{\frac{1}{k_j-m}}\le\cdots\le \bigg(s_n\prod_{i=1}^{k_{j-1}}\psi_i(n)^{-1}\prod_{i=1}^{d}\beta_i^n\bigg)^{\frac{1}{k_j-k_{j-1}}}.
		\end{split}
	\end{equation*}
\end{lem}
\begin{proof}
	Let $k_{j-1}\le \ell\le m-1$. We have
	\[\begin{split}
		&\bigg(s_n\prod_{i=1}^{\ell}\psi_i(n)^{-1}\prod_{i=1}^{d}\beta_i^n\bigg)^{k_j-(\ell+1)}\\
		=&\bigg(s_n\prod_{i=1}^{\ell+1}\psi_i(n)^{-1}\prod_{i=1}^{d}\beta_i^n\bigg)^{k_j-\ell}\bigg(s_n\prod_{i=1}^{\ell}\psi_i(n)^{-1}\prod_{i=1}^d\beta_i^n\bigg)^{-1}\psi_{\ell+1}(n)^{k_j-\ell}\\
		\ge&\bigg(s_n\prod_{i=1}^{\ell+1}\psi_i(n)^{-1}\prod_{i=1}^{d}\beta_i^n\bigg)^{k_j-\ell},
	\end{split}\]
	where the last inequality follows from the definition of $m$ (see \eqref{eq:defmkj}).
	Inequality raised to the power of $1/((k_l-\ell)(k_l-\ell-1))$ implies the lemma.
\end{proof}
Let
\begin{equation}\label{eq:omegan}
	\omega_n:=\beta_{m+1}^{-n}\bigg(s_n\prod_{i=1}^{m}\psi_i(n)^{-1}\prod_{i=1}^{d}\beta_i^n\bigg)^{\frac{1}{k_j-m}}=\bigg(s_n\prod_{i=1}^{m}\beta_i^n\psi_i(n)^{-1}\prod_{i=k_j+1}^d\beta_i^n\bigg)^{\frac{1}{k_j-m}},
\end{equation}
where the last equality follows from $\beta_{m+1}=\cdots=\beta_{k_j}$.
As a consequence of Lemma \ref{l:increase}, $\omega_n$ satisfies the following properties.
\begin{lem}\label{l:omeganp}
	For any sufficiently large $n$ with $n\in P$, we have
	\begin{enumerate}[\upshape(1)]
		\item For $1\le \ell\le m$, we have $\omega_n\le\beta_\ell^{-n}\psi_\ell(n)$;
		\item For $m+1\le \ell\le k_j$, we have $	\beta_\ell^{-n}\psi_\ell(n)\le \omega_n\le\beta_\ell^{-n}$;
		\item For $k_j+1\le \ell\le d$, we have $\omega_n\ge \beta_\ell^{-n}$.
	\end{enumerate}
\end{lem}
\begin{proof}
	(1) For any $n\in P$ and any $1\le \ell\le k_{j-1}$ with $k_{l-1}\le \ell-1<k_l$, by the definition of $m$ (see \eqref{eq:defmkj}),
	\[\psi_\ell(n)> \bigg(s_n\prod_{i=1}^{\ell-1}\psi_i(n)^{-1}\prod_{i=1}^{d}\beta_i^n\bigg)^{\frac{1}{k_l-\ell-1}}\ge s_n\prod_{i=1}^{d}\beta_i^n\ge n^{-2},\]
	where the last inequality follows from \eqref{eq:sufficientlarge}.
	Since $\ell\le k_{j-1}\le m$, we have $\beta_\ell<\beta_{m+1}$. So, for large $n$ with $n\in P$,
	\[\omega_n\le \beta_{m+1}^{-n}<\beta_\ell^{-n}n^{-2}<\beta_\ell^{-n}\psi_\ell(n).\]

	For $k_{j-1}<\ell\le m<k_j$, we have $\beta_\ell=\beta_{m+1}$. By Lemma \ref{l:increase} and \eqref{eq:defmkj},
	\[\begin{split}
		\omega_n&=\beta_{m+1}^{-n}\bigg(s_n\prod_{i=1}^{m}\psi_i(n)^{-1}\prod_{i=1}^{d}\beta_i^n\bigg)^{\frac{1}{k_j-m}}\\
		&\le \beta_\ell^{-n}\bigg(s_n\prod_{i=1}^{\ell-1}\psi_i(n)^{-1}\prod_{i=1}^{d}\beta_i^n\bigg)^{\frac{1}{k_j-\ell-1}}<\beta_{\ell}^{-n}\psi_\ell(n).
	\end{split}\]

	(2) For sufficiently large $n\in P$, since $\beta_{m+1}=\cdots=\beta_{k_j}$, by \eqref{eq:sufficientlarge} and \eqref{eq:defmkj},
	\[\begin{split}
		\beta_{m+1}^{-n}=\cdots=\beta_{k_j}^{-n}&\ge\beta_{m+1}^{-n}\bigg(s_n\prod_{i=1}^{m}\psi_i(n)^{-1}\prod_{i=1}^{d}\beta_i^n\bigg)^{\frac{1}{k_j-m}}\\
		&\ge \beta_{m+1}^{-n}\psi_{m+1}(n)\ge \cdots\ge\beta_{k_j}^{-n}\psi_{k_j}(n).
	\end{split}\]

	(3) For $k_j+1\le\ell\le d$, we have $\beta_{m+1}<\beta_\ell$. For large $n\in P$, by  \eqref{eq:sufficientlarge},
	\[\beta_{m+1}^{-n}\bigg(s_n\prod_{i=1}^{m}\psi_i(n)^{-1}\prod_{i=1}^{d}\beta_i^n\bigg)^{\frac{1}{k_j-m}}\ge \beta_{m+1}^{-n}n^{-2}>\beta_\ell^{-n},\]
	where we have used $\psi_i(n)\le 1$ for all $1\le i\le d$.
\end{proof}
Define a new $d$-tuple of functions $\Phi=(\phi_1,\dots,\phi_d)$ as follows: For $n\notin P$,
\[\phi_1(n)=\cdots=\phi_d(n)=0;\]
and for $n\in P$,
\[\phi_i(n)=\begin{dcases}
	\psi_i(n)/4& \text{if $1\le i\le m$},\\
	\beta_{m+1}^n\omega_n/4&\text{if $m+1\le i\le k_j$},\\
	1/4&\text{if $k_j+1\le i\le d$.}
\end{dcases}\]
Clearly, we have
\[\sum_{n=1}^{\infty}\prod_{i=1}^d\phi_i(n)=\sum_{n\in P} 4^{-d}s_n\prod_{i=1}^d\beta_i^n=\infty.\]
Therefore, by Theorem \ref{t:rectangleleb} the set $W_d(\Phi,\bm h)$ is of full Lebesgue measure.

Recall the definition of $\ce_n$ in \eqref{eq:wd}. Since we are dealing with the integral case, all the cylinders under consideration are full. Therefore, $\ce_n$ is independent of the $\Psi$ and $\Phi$. Therefore, by Lemma \ref{l:length},
\[\begin{split}
	W_d(\Phi,\bm h)&\subset\bigcap_{N=1}^\infty\bigcup_{n\in P:n\ge N}\bigcup_{\bz\in\mathcal E_n}\prod_{i=1}^{d}B\big(z_i,2\beta_i^{-n}\phi_i(n)\big)\\
	&\subset\bigcap_{N=1}^\infty\bigcup_{n\in P:n\ge N}\bigcup_{\bz\in\mathcal E_n}\bigg(\prod_{i=1}^mB\big(z_i,\beta_i^{-n}\psi_i(n)/2\big)\prod_{i=m+1}^{d}B(z_i,\omega_n)\bigg),
\end{split}\]
since by Lemma \ref{l:omeganp} (2) and (3), $\beta_i^{-n}\psi_i(n)\le\omega_n$ for $m+1\le i\le d$.
For any $\bz\in\mathcal E_n$, there is a finite collection $\mathcal B_n(\bz)$ of points such that for any $\by=(y_1,\dots,y_d)\in\mathcal B_n(\bz)$
\begin{equation}\label{eq:by}
	y_i\in B\big(z_i,\beta_i^{-n}\psi_i(n)/2\big)\text{ for $1\le i\le m$, and } y_i=z_i\text{ for $m+1\le i\le d$},
\end{equation}
and
\[\prod_{i=1}^mB\big(z_i,\beta_i^{-n}\psi_i(n)/2\big)\prod_{i=m+1}^{d}B(z_i,\omega_n)\subset\bigcup_{\by\in\mathcal B_n(\bz)}B(\by,\omega_n).\]
Then,
\[W_d(\Phi,\bm h)\subset\bigcap_{N=1}^\infty\bigcup_{n\in P:n\ge N}\bigcup_{\bz\in\mathcal E_n}\bigcup_{\by\in\mathcal B_n(\bz)}B(\by,\omega_n),\]
and so the rightmost $\limsup$ set defined by the collection
\begin{equation}\label{eq:limsupsetfullmea}
	\{B(\by,\omega_n):\by\in\mathcal B_n(\bz)\text{ with }\bz\in\mathcal E_n \text{ and }n\in P\}
\end{equation}
 of balls is of full Lebesgue measure. By \eqref{eq:wd},
\begin{equation}\label{eq:subsetmeasure}
	\bigcap_{N=1}^\infty\bigcup_{n\in P:n\ge N}\bigcup_{\bz\in\mathcal E_n}\bigcup_{\by\in\mathcal B_n(\bz)}\bigg(B(\by,\omega_n)\cap\bigcup_{\bz\in\mathcal E_n}\prod_{i=1}^{d}B\big(z_i,\beta_i^{-n}\psi_i(n)/2\big)\bigg)\subset W_d(\Psi,\bm h).
\end{equation}
By Theorem \ref{t:weaken} (1), to conclude that $W_d(\Psi,\bm h)$ has full Hausdorff $f$-measure, it is sufficient to prove that for all sufficiently large $n\in P$ and any $\by\in\cb_n(\bz)$ with $\bz\in\ce_n$,
\begin{equation}\label{eq:verify}
	\mathcal H_\infty^f\bigg(B(\by,\omega_n)\cap\bigcup_{\bz\in\mathcal E_n}\prod_{i=1}^{d}B\big(z_i,\beta_i^{-n}\psi_i(n)/2\big)\bigg)\gg \lm^d\big(B(\by,\omega_n)\big)\asymp \omega_n^d.
\end{equation}

From now on, fix a ball $B(\by,\omega_n)$ with $\by=(y_1,\dots,y_d)\in\mathcal B_n(\bz)$ and $\bz=(z_1,\dots,z_d)\in\mathcal E_n$. It follows from \eqref{eq:by} that:
\begin{enumerate}[\upshape(1)]
	\item For $1\le i\le m$, since $y_i\in B(z_i, \beta_i^{-n}\psi_i(n)/2) $ and $\omega_n\le\beta_i^{-n}\psi_i(n)$ (see Lemma \ref{l:omeganp} (1)),
	\[B(y_i,\omega_n)\cap\bigcup_{z_i\in\ce_{n,i}} B\big(z_i,\beta_i^{-n}\psi_i(n)/2\big)\]
	contains an interval of length $\omega_n$.
	\item For $m+1\le i\le k_j$, since $y_i=z_i$ and $\beta_i^{-n}\psi_i(n)\le \omega_n\le \beta_i^{-n}$ (see Lemma \ref{l:omeganp} (2)),
	\[B(y_i,\omega_n)\cap \bigcup_{z_i\in\mathcal E_{n,i}}B\big(z_i,\beta_i^{-n}\psi_i(n)/2\big)\]
	contains an interval of length $\beta_i^{-n}\psi_i(n)$.
	\item For $k_j+1\le i\le d$, since $y_i=z_i$ and $\beta_i^{-n}\psi_i(n)\le \beta_i^{-n}\le\omega_n$ (see Lemma \ref{l:omeganp} (3)),
	\[B(y_i,\omega_n)\cap \bigcup_{z_i\in\mathcal E_{n,i}}B\big(z_i,\beta_i^{-n}\psi_i(n)/2\big)\]
	contains $\asymp\omega_n \beta_i^n$ intervals, each of which is of length $\beta_i^{-n}\psi_i(n)$. Moreover, they are separated by a distance at least
	\begin{equation}\label{eq:separation}
		\beta_i^{-n}.
	\end{equation}
\end{enumerate}

Recall that $\ce_n=\ce_{n,1}\times\cdots\times \ce_{n,d}$. By the above discussion,
\begin{equation*}\label{eq:setcontent}
	B(\by,\omega_n)\cap\bigcup_{\bz\in\mathcal E_n}\prod_{i=1}^{d}B\big(z_i,\beta_i^{-n}\psi_i(n)/2\big)=\prod_{i=1}^d \bigcup_{z_i\in\ce_{n,i}}B(y_i,\omega_n)\cap B(z_i,\beta_i^{-n}\psi_i(n)/2)
\end{equation*}
contains one interval in $i$th direction with $1\le i\le k_j$, and contains $\asymp\omega_n \beta_i^n$ intervals in $i$th direction with $k_j+1\le i\le d$. Therefore, the above set contains
\[\asymp\prod_{i=k_j+1}^{d}\omega_n\beta_i^n\]
hyperrectangles, each of which is of lengths
\begin{equation}\label{eq:reclength}
	\underbrace{\omega_n,\dots,\omega_n}_{m},\beta_{m+1}^{-n}\psi_{m+1}(n),\dots,\beta_d^{-n}\psi_d(n).
\end{equation}

Let $\Delta$ be collection of these hyperrectangles.
Define a probability measure by
\begin{equation}\label{eq:mu}
	\mu=\frac{1}{\#\Delta}\sum_{R\in \Delta}\frac{\lm^d|_R}{\lm^d(R)}\asymp\frac{1}{\prod_{i=k_j+1}^{d}\omega_n\beta_i^n}\cdot\frac{1}{\omega_n^m\prod_{i=m+1}^{d}\beta_i^{-n}\psi_i(n)}\sum_{R\in \Delta}\lm^d|_R.
\end{equation}
Since $\omega_n<\beta_{m+1}^{-n}=\cdots=\beta_{k_j}^{-n}\le\beta_m^{-n}\cdots\le \beta_1^{-n}$ and $\omega_n\le \beta_i^{-n}\psi_i(n)$ for $1\le i\le m$, we define
\begin{equation}\label{eq:ca'}
	\ca_n':=\{\beta_{k_j+1}^{-n},\dots,\beta_d^{-n},\beta_{m+1}^{-n}\psi_{m+1}(n),\dots,\beta_d^{-n}\psi_d(n)\}.
\end{equation}

Now we estimate the $\mu$-measure of arbitrary balls $B(\bx,r)=\prod_{i=1}^dB(x_i,r)$ with $\bx\in R^*$ and $R^*\in\Delta$. The estimation is devided into three cases.

\noindent\textbf{Case 1:} $0<r<\min_{\tau\in\ca_n'}\tau=:\tau_{\min}$. It follows that $r<\beta_d^{-n}\le\cdots\le\beta_{k_j+1}^{-n}$. By the separation properties of $\Delta$ (see \eqref{eq:separation}),  we see that $B(\bx,r)$ only intersects the hyperrectangle $R^*$ to which $\bx$ belongs. By the definition of $\mu$ (see \eqref{eq:mu}),
\begin{align*}
	\mu\big(B(\bx,r)\big)&\asymp\frac{1}{\prod_{i=k_j+1}^{d}\omega_n\beta_i^n}\cdot\frac{1}{\omega_n^m\prod_{i=m+1}^{d}\beta_i^{-n}\psi_i(n)}\lm^d|_{R^*}\big(B(\bx,r)\big)\\
	&\ll\frac{1}{\prod_{i=k_j+1}^{d}\omega_n\beta_i^n}\cdot\frac{1}{\omega_n^m\prod_{i=m+1}^{d}\beta_i^{-n}\psi_i(n)}\cdot r^d\\
	&\ll\frac{1}{\prod_{i=k_j+1}^{d}\omega_n\beta_i^n}\cdot\frac{1}{\omega_n^m\prod_{i=m+1}^{d}\beta_i^{-n}\psi_i(n)}\cdot \frac{f(r)\tau_{\min}^{d}}{f(\tau_{\min})},
\end{align*}
where we use $f\preceq d$ in the last inequality.
Since $\tau_{\min}=\min_{\tau\in\mathcal A_{n}'}\tau=\min_{\tau\in\mathcal A_{n}}\tau$, we have
\[\ck_{n,1}(\tau_{\min})=\emptyset\quad\text{and}\quad\ck_{n,2}(\tau_{\min})=\{1,\dots,d\},\]
and so
\[s_n\le f(\tau_{\min})\prod_{i\in\ck_{n,1}(\tau_{\min})}\frac{\beta_i^{-n}}{\tau_{\min}}\prod_{i\in\ck_{n,2}(\tau_{\min})}\frac{\beta_i^{-n}\psi_i(n)}{\tau_{\min}}=f(\tau_{\min})\prod_{i=1}^d\frac{\beta_i^{-n}\psi_i(n)}{\tau_{\min}}\]
Thus,
\begin{equation}\label{eq:case1}
	\mu\big(B(\bx,r)\big)\ll\frac{1}{\prod_{i=k_j+1}^d\omega_n\beta_i^n}\cdot\frac{f(r)}{\omega_n^ms_n\prod_{i=1}^{m}\beta_i^n\psi_i(n)^{-1}}=\frac{f(r)}{\omega_n^d},
\end{equation}
where the equality follows from
\begin{align}
	\bigg(\prod_{i=k_j+1}^d\omega_n\beta_i^n\bigg)\omega_n^ms_n\prod_{i=1}^{m}\beta_i^n\psi_i(n)^{-1}&=\omega_n^{d-k_j+m}\bigg(s_n\prod_{i=1}^{m}\beta_i^n\psi_i(n)^{-1}\prod_{i=k_j+1}^d\beta_i^n\bigg)\notag\\
	&=\omega_n^d.\label{eq:identity}
\end{align}
Here we use \eqref{eq:omegan} in the last equality.

\noindent\textbf{Case 2:} Arrange the elements in $\ca_n'$ in non-descending order. Suppose that $\tau_\ell\le r<\tau_{\ell+1}$ with $\tau_\ell$ and $\tau_{\ell+1}$ are two consecutive terms in $\ca_n'$. By \eqref{eq:mu},
\begin{equation}\label{eq:step1}
	\mu\big(B(\bx,r)\big)
	\ll \frac{1}{\prod_{i=k_j+1}^{d}\omega_n\beta_i^n}\cdot\frac{1}{\omega_n^m\prod_{i=m+1}^{d}\beta_i^{-n}\psi_i(n)}\sum_{R:R\cap B(\bx,r)\ne\emptyset}\lm^d|_R\big(B(\bx,r)\big).
\end{equation}

Next, we estimate the number of hyperrectangles $R\in\Delta$ that intersects $B(\bx,r)$ and the $\lm^d$-measure of the corresponding intersection $B(\bx,r)\cap R$. For this purpose, define the sets
\[\ck_{n,1}(\tau_\ell)=\{i:\beta_i^{-n}\le \tau_\ell\},\ \ck_{n,2}(\tau_{\ell+1})=\{i:\beta_i^{-n}\psi_i(n)\ge \tau_{\ell+1}\},\]
\[\ck_{n,3}=\{1,\dots,d\}\setminus\big(\ck_{n,1}(\tau_\ell)\cup\ck_{n,2}(\tau_{\ell+1})\big).\]
It should be noticed the above sets we are going to use are slightly different to those given in Theorem \ref{t:rectangle}. The estimation is built upon the following observations:

\noindent{\textbf{Observation A}}: in the directions $i\in\ck_{n,1}(\tau_\ell)$,
\[r\ge \tau_\ell\ge \beta_i^{-n}.\]
So the total number of hyperrectangles that intersecting $B(\bx,r)$ along the $i$th direction is majorized by
\begin{equation}\label{eq:Ni}
	\ll r\beta_i^n.
\end{equation}
Since $\tau_\ell< \tau_{\ell+1}\le\max_{\tau\in\ca'_n}\tau< \beta_{m+1}^{-n}$ (by the definition of $\ca_n'$), we have $i>m+1$. Then, by \eqref{eq:reclength}, the length of $R$ in $i$th direction is $\beta_i^{-n}\psi_i(n)$. Therefore, for any $R\in\Delta$, its intersection with $B(\bx,r)$ in the $i$th direction is of length at most
\begin{equation}	\label{eq:mea1}
		\ll \beta_i^{-n}\psi_i(n).
\end{equation}

\noindent{\textbf{Observation B}}: in the directions $i\in\ck_{n,2}(\tau_{\ell+1})$,
\[r<\tau_{\ell+1}\le \beta_i^{-n}\psi_i(n)\le \beta_i^{-n}.\]
The ball $B(\bx,r)$ only intersect one hyperrectangle in the $i$th direction. Consequently, for any
$R\in\Delta$, the length of its intersection with $B(\bx,r)$ along $i$th direction is at most
\begin{equation}\label{eq:mea2}
	\ll r.
\end{equation}

\noindent{\textbf{Observation C}}: in the directions $i\in\ck_{n,3}$, since $\tau_\ell$ and $\tau_{\ell+1}$ are two consecutive terms in $\ca_n'$, we have $\beta_i^{-n}\ge\tau_{\ell+1}>\tau_\ell $, and so
\[\tau_\ell\le r<\tau_{\ell+1}\le \beta_i^{-n}.\]
This means that $B(\bx,r)$ only intersects one hyperrectangle in the $i$th direction. Hence, the length of the intersection $B(\bx,r)\cap R$ in $i$th direction is at most
\begin{equation}\label{eq:mea3}
	\ll \beta_i^{-n}\psi_i(n).
\end{equation}
Thus, by \eqref{eq:Ni}--\eqref{eq:mea3},
\[\begin{split}
	&\sum_{R:R\cap B(\bx,r)\ne\emptyset}\lm^d|_R\big(B(\bx,r)\big)\\
	\ll & \prod_{i\in\ck_{n,1}(\tau_\ell)}r\psi_i(n)\prod_{i\in\ck_{n,2}(\tau_{\ell+1})}r\prod_{i\in\ck_{n,3}}\beta_i^{-n}\psi_i(n)\\
	=&\prod_{i=1}^d\beta_i^{-n}\psi_i(n)\prod_{i\in\ck_{n,1}(\tau_\ell)}\frac{r}{\beta_i^{-n}} \prod_{i\in\ck_{n,2}(\tau_{\ell+1})}\frac{r}{\beta_i^{-n}\psi_i(n)}.
\end{split}\]
This together with \eqref{eq:step1} gives
\[\mu\big(B(x,r)\big)
\ll\frac{1}{\prod_{i=k_j+1}^{d}\omega_n\beta_i^n}\cdot\frac{1}{\omega_n^m\prod_{i=1}^{m}\beta_i^{n}\psi_i(n)^{-1}}\prod_{i\in\ck_{n,1}(\tau_\ell)}\frac{r}{\beta_i^{-n}} \prod_{i\in\ck_{n,2}(\tau_{\ell+1})}\frac{r}{\beta_i^{-n}\psi_i(n)}.\]
Write $k:=\#(\ck_{n,1}(\tau_\ell)\cup\ck_{n,2}(\tau_{\ell+1}))$. By the condition on $f$ (see Theorem \ref{t:rectangle}), either $k\preceq f$ or $f\preceq k$. Therefore, since $\tau_\ell\le r<\tau_{\ell+1}$, by \eqref{eq:prec} either
\[\frac{r^k}{f(r)}\le \frac{\tau_{\ell}^k}{f(\tau_{\ell})}\quad\text{or}\quad \frac{f(\tau_{\ell+1})}{\tau_{\ell+1}^k}\le \frac{f(r)}{r^k}.\]
Suppose we are in the formal case. Substituting $k=\#(\ck_{n,1}(\tau_\ell)\cup\ck_{n,2}(\tau_{\ell+1}))$, we have
\[\begin{split}
	&\mu\big(B(x,r)\big)\\
	\ll&\frac{1}{\prod_{i=k_j+1}^{d}\omega_n\beta_i^n}\cdot\frac{1}{\omega_n^m\prod_{i=1}^{m}\beta_i^{n}\psi_i(n)^{-1}}\cdot\frac{f(r)}{f(\tau_\ell)}\prod_{i\in\ck_{n,1}(\tau_\ell)}\frac{\tau_\ell}{\beta_i^{-n}} \prod_{i\in\ck_{n,2}(\tau_{\ell+1})}\frac{\tau_\ell}{\beta_i^{-n}\psi_i(n)}.
\end{split}\]
We claim that
\begin{equation}\label{eq:claim}
	\prod_{i\in\ck_{n,2}(\tau_{\ell+1})}\frac{\tau_\ell}{\beta_i^{-n}\psi_i(n)}=\prod_{i\in\ck_{n,2}(\tau_\ell)}\frac{\tau_\ell}{\beta_i^{-n}\psi_i(n)}.
\end{equation}
Clearly, $\ck_{n,2}(\tau_{\ell+1})\subset\ck_{n,2}(\tau_\ell)$. If $\ck_{n,2}(\tau_{\ell+1})=\ck_{n,2}(\tau_\ell)$, then there is nothing to be proved. Now, suppose that $\ck_{n,2}(\tau_{\ell+1})\subsetneqq\ck_{n,2}(\tau_\ell)$. Since $\tau_\ell<\tau_{\ell+1}$ are two consecutive terms in $\ca_n'$, one has
\[\tau_\ell=\beta_{i_0}^{-n}\psi_{i_0}(n) \quad\text{for some $i_0$},\]
which implies $\ck_{n,2}(\tau_\ell)=\ck_{n,2}(\tau_{\ell+1})\cup\{i_0\}$. Therefore, \eqref{eq:claim} follows. By the definiton of $s_n$,
\[\begin{split}
	&\mu\big(B(x,r)\big)\\
	\ll&\frac{1}{\prod_{i=k_j+1}^{d}\omega_n\beta_i^n}\cdot\frac{1}{\omega_n^m\prod_{i=1}^{m}\beta_i^{n}\psi_i(n)^{-1}}\cdot\frac{f(r)}{f(\tau_\ell)}\prod_{i\in\ck_{n,1}(\tau_\ell)}\frac{\tau_\ell}{\beta_i^{-n}} \prod_{i\in\ck_{n,2}(\tau_{\ell+1})}\frac{\tau_\ell}{\beta_i^{-n}\psi_i(n)}\\
	=&\frac{1}{\prod_{i=k_j+1}^{d}\omega_n\beta_i^n}\cdot\frac{1}{\omega_n^m\prod_{i=1}^{m}\beta_i^{n}\psi_i(n)^{-1}}\cdot\frac{f(r)}{f(\tau_\ell)}\prod_{i\in\ck_{n,1}(\tau_\ell)}\frac{\tau_\ell}{\beta_i^{-n}} \prod_{i\in\ck_{n,2}(\tau_\ell)}\frac{\tau_\ell}{\beta_i^{-n}\psi_i(n)}\\
	\le&\frac{1}{\prod_{i=k+1}^d\omega_n\beta_i^n}\cdot\frac{1}{\omega_n^m\prod_{i=1}^{m}\beta_i^{n}\psi_i(n)^{-1}}\cdot\frac{f(r)}{s_n}=\frac{f(r)}{\omega_n^d},
\end{split}\]
where the last equality follows from the same reason as \eqref{eq:identity}.

The other case can be proved in a similar way with some obvious modifications.

\noindent\textbf{Case 3:} $\tau_{\max}:=\max_{\tau\in\ca'_n}\tau\le r<\omega_n$. In this case, the ball $B(\bx,r)$ is sufficently large so that $r\ge \beta_i^{-n}$ for $k_j+1\le i\le d$.
For any $i\ge k_j+1$, the total number of hyperrectangles that intersecting $B(\bx,r)$ along the $i$th direction is majorized by
\begin{equation}\label{eq:Ni2}
\ll r\beta_i^n.
\end{equation}
For $1\le i\le m$, by \eqref{eq:reclength} the length of $R\in\Delta$ in $i$th direction is $\omega_n>r$. While for $m+1\le i\le d$, by \eqref{eq:reclength} the length of $R\in\Delta$ in $i$th direction is $\beta_i^{-n}\psi_i(n)\le \tau_{\max}\le r$.
By the definition of $\mu$, one has
\begin{align}
	\mu\big(B(\bx,r)\big)&\asymp \frac{1}{\prod_{i=k_j+1}^{d}\omega_n\beta_i^n}\cdot\frac{1}{\omega_n^m\prod_{i=m+1}^d\beta_i^{-n}\psi_i(n)}\sum_{R:R\cap B(\bx,r)\ne\emptyset}\lm^d|_R\big(B(\bx,r)\big)\notag\\
	&\ll \frac{1}{\prod_{i=k_j+1}^{d}\omega_n\beta_i^n}\cdot\frac{1}{\omega_n^m\prod_{i=m+1}^{d}\beta_i^{-n}\psi_i(n)} \prod_{i=k_j+1}^dr\beta_i^n\prod_{i=1}^{m}r\prod_{i=m+1}^{d}\beta_i^{-n}\psi_i(n)\notag\\
	&=\frac{r^{d-k_j+m}}{\omega_n^{d-k_j+m}}.\label{eq:case3}
\end{align}

If $f\preceq (d-k_j+m)$, then we have
\[\frac{f(1)}{1}\le \frac{f(r)}{r^{d-k_j+m}}\quad\Longrightarrow\quad r^{d-k_j+m}\ll f(r),\]
since $0<r<1$.
Therefore,
\[\mu\big(B(\bx,r)\big)\ll \frac{f(r)}{\omega_n^{d-k_j+m}}\le \frac{f(r)}{\omega_n^d},\]
where have used $\omega_n\le 1$.

If $(d-k_j+m)\preceq f$, then we have
\[\frac{r^{d-k_j+m}}{f(r)}\le \frac{\tau_{\max}^{d-k_j+m}}{f(\tau_{\max})},\]
since $r\ge \tau_{\min}$.
Therefore, by \eqref{eq:case3}
\[\begin{split}
	\mu\big(B(\bx,r)\big)&\ll \frac{r^{d-k_j+m}}{\omega_n^{d-k_j+m}}\le\frac{f(r)}{\omega_n^{d-k_j+m}}\cdot \frac{\tau_{\max}^{d-k_j+m}}{f(\tau_{\max})}.
\end{split}\]
Note that in this case, $\ck_{n,1}(\tau_{\max})=\{k_j+1,\dots,d\}$ and $\ck_{n,2}(\tau_{\max})=\{1,\dots,m\}$. So,
\[\begin{split}
	s_n&\le f(\tau_{\max})\prod_{i\in\ck_{n,1}(\tau_{\max})}\frac{\beta_i^{-n}}{\tau_{\max}}\prod_{i\in\ck_{n,2}(\tau_{\max})}\frac{\beta_i^{-n}\psi_i(n)}{\tau_{\max}}\\
	&=f(\tau_{\max})\prod_{i=k_j+1}^{d}\frac{\beta_i^{-n}}{\tau_{\max}}\prod_{i=1}^m\frac{\beta_i^{-n}\psi_i(n)}{\tau_{\max}},
\end{split}\]
or equivalently,
\[\frac{\tau_{\max}^{d-k_j+m}}{f(\tau_{\max})}\le \frac{1}{s_n\prod_{i=k_j+1}^{d}\beta_i^n\prod_{i=1}^{m}\beta_i^n\psi_i(n)^{-1}}=\frac{1}{\omega_n^{k_j-m}},\]
where the equality follows from \eqref{eq:omegan}.
It holds that
\[\mu\big(B(\bx,r)\big)\ll\frac{f(r)}{\omega_n^{d-k_j+m}}\cdot\frac{1}{\omega_n^{k_j-m}}=\frac{f(r)}{\omega_n^{d}}.\]

Now, we are able to completing the proof of the divergence part of Theorem \ref{t:rectangle}.
\begin{proof}[Completing the proof of Theorem \ref{t:rectangle}: Divergence part]
     Note that $f$ is supposed to satisfy $f\preceq d$. The case
     \[\lim_{r\to 0^+} f(r)/r^d<\infty\]
     has been proven in Section \ref{ss:reduction}. For $f\prec d$, i.e.
      \[\lim_{r\to 0^+} f(r)/r^d=\infty,\]
      by the discussion of Cases 1--3 above, we see that there is probability measure $\mu$ supported on
      \[B(\by,\omega_n)\cap\bigcup_{\bz\in\mathcal E_n}\prod_{i=1}^{d}B\big(z_i,\beta_i^{-n}\psi_i(n)/2\big)\]
      such that for any ball $B(\bx,r)$,
      \[\mu\big(B(\bx,r)\big)\ll\frac{f(r)}{\omega_n^d}.\]
      By the mass distribution principle, we have
      \[\mathcal H_\infty^f\bigg(B(\by,\omega_n)\cap\bigcup_{\bz\in\mathcal E_n}\prod_{i=1}^{d}B\big(z_i,\beta_i^{-n}\psi_i(n)/2\big)\bigg)\gg \omega_n^d\asymp\lm^d\big(B(\by,\omega_n)\big),\]
      which verifies \eqref{eq:verify}. Since the $\limsup$ set defined by
      \[\{B(\by,\omega_n):\by\in\mathcal B_n(\bz)\text{ with }\bz\in\mathcal E_n \text{ and }n\in P\}\]
      has full Lebesgue measure (see \eqref{eq:limsupsetfullmea}), by \eqref{eq:subsetmeasure} and Theorem \ref{t:weaken} (1), we conclude that
      \[\hm^f\big(W_d(\Psi,\bm h)\big)=\hm^f([0,1]^d).\qedhere\]
\end{proof}
\section{Proof of Theorem \ref{t:example}}\label{s:example}
In this section, we prove Theorem \ref{t:example} with the help of Theorem \ref{t:rectangle}. Recall that
\[W_2^*(t):=\{\bx\in[0,1)^2:|T_2^nx_1|<e^{-nt}\text{ and }|T_3^nx_2|<e^{-n^2}\text{ for i.m.\,$n$}\},\]
and the Hausdorff dimension of this set (see \eqref{eq:hdimW}) is
\[\hdim W_2^*(t)=\min\bigg(1,\frac{\log 2+\log 3}{\log 2+t}\bigg).\]

\begin{proof}[Proof of Theorem \ref{t:example}] Clearly, for any $n\ge 1$,
	\[\ca_n=\{2^{-n},3^{-n},2^{-n}e^{-nt},3^{-n}e^{-n^2}\}.\]
	(1) Suppose that $t>\log 3$. Then, $\hdim W_2^*(t)<1$. For any dimension function $f$ with $1\preceq f$, we have
	\[\hm^f\big(W_2^*(t)\big)\ll \hm^1\big(W_2^*(t)\big)=0.\]

	Next, we focus on the case $f\prec 1$. In view of Theorem \ref{t:rectangle}, the main task is to simplify the expression for the minimum
	\begin{equation}\label{eq:minsn}
		\min\biggl\{ f(2^{-n})\frac{3^{-n}}{2^{-n}},f(3^{-n}),f(2^{-n}e^{-nt}),f(3^{-n}e^{-n^2})\frac{2^{-n}e^{-nt}}{3^{-n}e^{-n^2}}\biggr\}.
	\end{equation}
	Using the definition of $f\prec 1$, and noting that $2^{-n}e^{-nt}>3^{-n}e^{-n^2}$ and $2^{-n}>3^{-n}$, we obtain
	\[\frac{f(2^{-n}e^{-nt})}{2^{-n}e^{-nt}}\le \frac{f(3^{-n}e^{-n^2})}{3^{-n}e^{-n^2}}\quad\Longrightarrow\quad f(2^{-n}e^{-nt})\le f(3^{-n}e^{-n^2})\frac{2^{-n}e^{-nt}}{3^{-n}e^{-n^2}},\]
	and
	\[\frac{f(2^{-n})}{2^{-n}}\le \frac{f(3^{-n})}{3^{-n}}\quad\Longrightarrow\quad f(2^{-n})\frac{3^{-n}}{2^{-n}}\le f(3^{-n}).\]
	Therefore, the minimum in \eqref{eq:minsn} can be simplified to
	\[\min\biggl\{f(2^{-n}e^{-nt}), f(2^{-n})\frac{3^{-n}}{2^{-n}}\biggr\},\]
	which implies that the Hausdorff $f$-measure of $W_2^*(t)$ is zero or full according as
	\begin{equation}\label{eq:examseries}
		\sum_{n=1}^{\infty}\min\{f(2^{-n}e^{-nt})6^n, f(2^{-n})4^n\}
	\end{equation}
	converges or not. On the other hand, since $f\prec 1$,
	\[\frac{f(1)}{1}\le \frac{f(2^{-n})}{2^{-n}}\quad\Longrightarrow\quad f(2^{-n})\gg 2^{-n}\quad\Longrightarrow\quad f(2^{-n})4^n\gg 2^n.\]
	This means that the convergence or divergence of the series in \eqref{eq:examseries} is equivalent to that of
	\[\sum_{n=1}^{\infty}f(2^{-n}e^{-nt})6^n.\]
	Using Theorem \ref{t:rectangle}, we conclude item (1).

	(2) Suppose that $t\le\log 3$. Then, $\hdim W_2^*(t)=1$. By \cite[Remark 9]{He24a} we have
	\[\hm^1\big(W_2^*(t)\big)=\hm^1([0,1]^2)=\infty.\]
	Hence, for any dimension function $f$ with $f\prec 1$,
	\[\hm^f\big(W_2^*(t)\big)\ge \hm^1\big(W_2^*(t)\big)=\infty=\hm^f([0,1]^2).\]

	Now, let $f$ be a dimenison function such that $1\preceq f\preceq(1+\log2/\log3)$. There are two natural cases needed to be considered.

	First, suppose that $t\le \log2/\log3$, meaning that $2^{-n}e^{-nt}\ge 3^{-n}$ for all $n\ge 1$. In analogy to the proof of item (1), the Hausdorff $f$-measure of $W_2^*(t)$ is completely determined by the convergence or divergence of the series
	\begin{equation}\label{eq:minsns}
		\sum_{n=1}^{\infty}\min\biggl\{f(2^{-n})\frac{3^{-n}}{2^{-n}},f(3^{-n})\frac{2^{-n}e^{-nt}}{3^{-n}},f(2^{-n}e^{-nt})\frac{3^{-n}}{2^{-n}e^{-nt}},f(3^{-n}e^{-n^2})\frac{2^{-n}e^{-nt}}{3^{-n}e^{-n^2}} \biggr\}\cdot 6^n.
	\end{equation}
	Using $1\preceq f$, and noting that  $2^{-n}>2^{-n}e^{-nt}$ and $3^{-n}>3^{-n}e^{-n^2}$, we obtain
	\[\frac{2^{-n}}{f(2^{-n})}\le \frac{2^{-n}e^{-nt}}{f(2^{-n}e^{-nt})} \quad\Longrightarrow\quad f(2^{-n}e^{-nt})\frac{3^{-n}}{2^{-n}e^{-nt}}\le f(2^{-n})\frac{3^{-n}}{2^{-n}},\]
	and
	\[\frac{3^{-n}}{f(3^{-n})}\le \frac{3^{-n}e^{-n^2}}{f(3^{-n}e^{-n^2})}\quad\Longrightarrow\quad f(3^{-n}e^{-n^2})\frac{2^{-n}e^{-nt}}{3^{-n}e^{-n^2}}\le f(3^{-n})\frac{2^{-n}e^{-nt}}{3^{-n}}.\]
	Therefore, the series in \eqref{eq:minsns} equals to
\begin{equation}\label{eq:examseriess}
	\sum_{n=1}^{\infty}\min\biggl\{f(3^{-n}e^{-n^2})\frac{2^{-n}e^{-nt}}{3^{-n}e^{-n^2}},f(2^{-n}e^{-nt})\frac{3^{-n}}{2^{-n}e^{-nt}}\biggr\}\cdot 6^n.
\end{equation}
	By the assumption $1\preceq f$ and the fact
	$2^{-n}e^{-nt}\ge 3^{-n}$,
		\[\frac{2^{-n}e^{-tn}}{f(2^{-n}e^{-tn})}\le \frac{3^{-n}}{f(3^{-n})}\quad\Longrightarrow\quad f(2^{-n}e^{-nt})\frac{3^{-n}}{2^{-n}e^{-nt}}\ge f(3^{-n}),\]
	and then by $f\preceq (1+\log 2/\log 3)$,
	\[\frac{f(1)}{1^{1+\log2/\log 3}}\le \frac{f(3^{-n})}{3^{-n(1+\log2/\log 3)}}=\frac{f(3^{-n})}{6^{-n}}\quad\Longrightarrow\quad  f(3^{-n})\gg 6^{-n}.\]
	These two inequalities yield
	\[f(2^{-n}e^{-nt})\frac{3^{-n}}{2^{-n}e^{-nt}}\cdot 6^n\gg 1,\]
	which means that the convergence or divergence of the series in \eqref{eq:examseriess} is equivalent to that of
	\[\sum_{n=1}^{\infty}f(3^{-n}e^{-n^2})9^ne^{n^2-nt}.\]

	The remaining case that $t>\log2/\log3$ follows in the same manner, with some straightforward modifications. We omit the details.
\end{proof}

\section{Proof of Theorem \ref{t:multiplicative} (1)}\label{s:onedim} This section is devoted to providing a dichotomy law for the Hausdorff measure of $W_1(\psi,h)$, which will be crucial for the proof of the divergence part of Theorem \ref{t:multiplicative} (2). For the moment, suppose that $d=1$ and $\beta>1$ (not necessarily an integer). We first develop a dichotomy law for the Lebesgue measure of the auxilliary set
\[\begin{split}
	\tilde{W}_1(\psi,h):=\bigcap_{N=1}^\infty\bigcup_{n=N}^\infty\bigcup_{\be_n\in\Lambda_\beta^n}\{x\in I_{n,\beta}(\be_n):|T_\beta^nx-h(x)|<\psi(n)\}=:\bigcap_{N=1}^\infty\bigcup_{n=N}^\infty \tilde E_n.
\end{split}\]
which subsequently allows us to apply the mass transference principle to derive the Hausdorff measure of $W_1(\psi,h)$. It is easy to see that
\[\tilde W(\psi,h)\subset W_1(\psi,h).\]
\begin{lem}\label{l:onedim}
	Let $\psi:\N\to\R^+$. We have
	\[\lm^1\big(\tilde W_1(\psi,h)\big)=\begin{cases}
		0&\text{if $\sum_{n=1}^\infty\psi(n)<\infty$},\\
		1&\text{if $\sum_{n=1}^\infty\psi(n)=\infty$}.
	\end{cases}\]
\end{lem}
Before giving the proof, we state and prove some necessary lemmas.
\begin{lem}[Chung-Erd\"os inequality]
	For measurable sets $A_1,\dots,A_N$,
	\[\lm^1(A_1\cup\cdots\cup A_N)\ge \frac{\big(\sum_{n=1}^{N}\lm^1(A_n)\big)^2}{\sum_{n,m=1}^N\lm^1(A_n\cap A_m)}.\]
\end{lem}
In view of Chung-Erd\"os inequality, the key to obtain Lemma \ref{l:onedim} is to estimate $\lm^1(\tilde E_n)$ ($n\ge 1$) and the correlations $\lm^1(\tilde E_m\cap\tilde E_n)$ ($1\le m<n$), which will be presented in the following two lemmas, respectively.
\begin{lem}\label{l:tildeEn}
	For any full sequence $\be_n\in \Lambda_\beta^n$, we have
	\[I_{n,\beta}(\be_n)\cap \tilde E_n\asymp \beta^{-n}\psi(n).\]
	In particular, let $\be_k\in\Lambda_\beta^k$ be a full sequence with $k\le n$. Then,
	\[\lm^1(I_{k,\beta}(\be_k)\cap \tilde E_n)\asymp |I_{k,\beta}(\be_k)|\psi(n).\]
\end{lem}
\begin{proof}
	The first point of the lemma follows from Lemma \ref{l:length}. To conclude the second point, it suffices to notice from Lemma \ref{l:Li} that the number of $n$th level full cylinder inside $I_{k,\beta}(\be_k)$ is approximately $|I_{k,\beta}(\be_k)|\beta^n$.
\end{proof}
\begin{lem}\label{l:correlation}
	Let $\be_k\in\Lambda_\beta^k$ be a full sequence. For any $M>N\ge k$, we have
	\[\sum_{n,m=N}^M\lm^1\big(I_{k,\beta}(\be_k)\cap \tilde E_m\cap\tilde E_n\big)\ll |I_{k,\beta}(\be_k)|\cdot \bigg(\bigg(\sum_{n=N}^M\psi(n)\bigg)^2+\sum_{n=N}^M\psi(n)\bigg).\]
\end{lem}
\begin{proof}
	Note that
	\begin{align}
		&\sum_{n,m=N}^M\lm^1\big(I_{k,\beta}(\be_k)\cap \tilde E_m\cap\tilde E_n\big)\notag\\
		=&2\sum_{n=N}^M\sum_{m=N}^{n-1}\lm^1\big(I_{k,\beta}(\be_k)\cap \tilde E_m\cap\tilde E_n\big)+2\sum_{n=N}^M\lm^1\big(I_{k,\beta}(\be_k)\cap \tilde E_n\big)\notag\\
		\asymp&\sum_{n=N}^M\sum_{m=N}^{n-1}\lm^1\big(I_{k,\beta}(\be_k)\cap \tilde E_m\cap\tilde E_n\big)+|I_{k,\beta}(\be_k)|\sum_{n=N}^M\psi(n),\label{eq:correlation}
	\end{align}
	where we have used Lemma \ref{l:tildeEn} in the last inequality.
	Let $M\ge n>m\ge N\ge k$. The estimation of $\lm^1\big(I_{k,\beta}(\be_k)\cap \tilde E_m\cap\tilde E_n\big)$ is divided into two cases.

	\noindent{\bf Case 1:} $\beta^{-m}\psi(m)\ge\beta^{-n}$. By Lemmas \ref{l:renyi}, \ref{l:Li} and  \ref{l:length}, for any $\ell\ge k$, $I_{k,\beta}(\be_k)\cap\tilde E_\ell$ is contained in $\asymp |I_{k,\beta}(\be_k)|\beta^\ell$ intervals, each of which has length $\asymp \beta^{-\ell}\psi(\ell)$. Let $\mathcal C_\ell$ denote the collection of these intervals. For any $I^{(m)}\in \mathcal C_m$, it intersects at most
	\[\ll\big\lfloor|I^{(m)}|\beta^n\big\rfloor+1\asymp \beta^{n-m}\psi(m)\]
	$n$th level full cylinders, since $|I^{(m)}|\asymp\beta^{-m}\psi(m)\ge \beta^{-n}$. Therefore, for any $I^{(m)}\in \mathcal C_m$,
	\[\begin{split}
		\lm^1(I^{(m)}\cap \tilde E_n)&=\sum_{\substack{I^{(n)}\in \mathcal C_n\\ I^{(n)}\cap I^{(m)}\ne\emptyset}}\lm^1(I^{(m)}\cap I^{(n)})\ll \sum_{\substack{I^{(n)}\in \mathcal C_n\\ I^{(n)}\cap I^{(m)}\ne\emptyset}} \beta^{-n}\psi(n)\\
		&\ll \beta^{-m}\psi(m)\psi(n).
	\end{split}\]
	Summing over all $I^{(m)}\in \mathcal C_m$, we have
	\[\begin{split}
		\lm^1\big(I_{k,\beta}(\be_k)\cap \tilde E_m\cap\tilde E_n\big)&=\sum_{I^{(m)}\in \mathcal C_m}\lm^1(I^{(m)}\cap \tilde E_n)\ll\#\mathcal C_m\cdot \beta^{-m}\psi(m)\psi(n)\\
		&\asymp |I_{k,\beta}(\be_k)|\psi(m)\psi(n).
	\end{split} \]

	\noindent{\bf Case 2:} $\beta^{-m}\psi(m)<\beta^{-n}$. Each $I^{(m)}\in \mathcal C_m$ intersects at most
	\[\ll 1\]
	 $n$th level full cylinders. Therefore,
	\[\lm^1(I^{(m)}\cap \tilde E_n)\ll \beta^{-n}\psi(n).\]
	Summing over all $I^{(m)}\in \mathcal C_m$, we have
	\[\lm^1\big(I_{k,\beta}(\be_k)\cap \tilde E_m\cap\tilde E_n\big)\ll\#\mathcal C_m\cdot \beta^{-n}\psi(n)\asymp |I_{k,\beta}(\be_k)|\beta^{m-n}\psi(n). \]

	It follows from the above discussion that
	\[\begin{split}
		&\sum_{n=N}^M\sum_{m=N}^{n-1}\lm^1\big(I_{k,\beta}(\be_k)\cap \tilde E_m\cap\tilde E_n\big)\\
		=& \sum_{n=N}^M\sum_{\substack{m=N\\ \beta^{-m}\psi(m)\ge\beta^{-n}}}^{n-1}\lm^1\big(I_{k,\beta}(\be_k)\cap \tilde E_m\cap\tilde E_n\big)+\sum_{n=N}^M\sum_{\substack{m=N\\ \beta^{-m}\psi(m)<\beta^{-n}}}^{n-1}\lm^1\big(I_{k,\beta}(\be_k)\cap \tilde E_m\cap\tilde E_n\big)\\
		\ll& \sum_{n=N}^M\sum_{\substack{m=N\\ \beta^{-m}\psi(m)\ge\beta^{-n}}}^{n-1}|I_{k,\beta}(\be_k)|\psi(m)\psi(n)+\sum_{n=N}^M\sum_{\substack{m=N\\ \beta^{-m}\psi(m)<\beta^{-n}}}^{n-1}|I_{k,\beta}(\be_k)|\beta^{m-n}\psi(n)\\
		\le& \sum_{n=N}^M\sum_{m=N}^{n-1}|I_{k,\beta}(\be_k)|\psi(m)\psi(n)+\sum_{n=N}^M\sum_{m=N}^{n-1}|I_{k,\beta}(\be_k)|\beta^{m-n}\psi(n)\\
		\ll &|I_{k,\beta}(\be_k)|\cdot \bigg(\bigg(\sum_{n=N}^M\psi(n)\bigg)^2+\sum_{n=N}^M\psi(n)\bigg).
	\end{split}\]
	This together with \eqref{eq:correlation} completes the proof.
\end{proof}
Now, we are able to prove Lemma \ref{l:onedim}.
\begin{proof}[Proof of Lemma \ref{l:onedim}]
	The convergence part of Lemma \ref{l:onedim} follows directly from Lemmas \ref{l:tildeEn}, since
	\[\lm^1\big(\tilde W_1(\psi,h)\big)\le \limsup_{N\to\infty}\sum_{n=N}^\infty\lm^1(\tilde E_n)\asymp\limsup_{N\to\infty}\sum_{n=N}^\infty\psi(n)=0\]
	provided $\sum_{n=1}^\infty\psi(n)<\infty$.

	Now suppose that $\sum_{n=1}^\infty\psi(n)=\infty$. Let $\be_k\in\Lambda_\beta^k$ be a full sequence. By Chung-Erd\"os inequality and Lemma \ref{l:correlation}, we have
	\[\begin{split}
		&\lim_{N\to\infty}\lm^1\bigg( I_{k,\beta}(\be_k)\cap\bigcup_{n=N}^\infty\tilde E_n\bigg)
		=\lim_{N\to\infty}\lim_{M\to\infty}\lm^1\bigg(\bigcup_{n=N}^M I_{k,\beta}(\be_k)\cap\tilde E_n\bigg)\\
		\ge & \lim_{N\to\infty}\lim_{M\to\infty}\frac{\big(\sum_{n=N}^{M}\lm^1(I_{k,\beta}(\be_k)\cap\tilde E_n)\big)^2}{\sum_{n,m=N}^M\lm^1\big(I_{k,\beta}(\be_k)\cap\tilde E_m\cap\tilde E_n\big)}\\
		\gg&\lim_{N\to\infty}\lim_{M\to\infty}\frac{|I_{k,\beta}(\be_k)|^2\big(\sum_{n=N}^{M}\psi(n)\big)^2}{|I_{k,\beta}(\be_k)|\cdot \big(\big(\sum_{n=N}^M\psi(n)\big)^2+\sum_{n=N}^M\psi(n)\big)}\gg|I_{k,\beta}(\be_k)|,
	\end{split}\]
	where the unspecify constant does not depend on $\be_k$. Since the $\limsup$ set defined by all full cylinders has full Lebesgue measure (see Lemma \ref{l:union=[0,1]}), by Theorem \ref{t:weaken} (2), we conclude that
	\[1=\lm^1\bigg(\limsup_{N\to\infty}\bigcup_{n=N}^\infty\tilde E_n\bigg)=\lm^1\Big(\limsup_{n\to\infty}\tilde E_n\Big)=\lm^1\big(\tilde W_1(\psi,h)\big).\qedhere\]
\end{proof}
Recall that in Theorem \ref{t:multiplicative} (1), $g$ is a dimension function such that $g\preceq 1$.
\begin{proof}[Proof of Theorem \ref{t:multiplicative} (1)]
	First, suppose that $\sum_{n=1}^\infty\beta^ng\big(\beta^{-n}\psi(n)\big)<\infty$. Since
	\[W_1(\psi,h)=\bigcap_{N=1}^\infty\bigcup_{n=N}^\infty\bigcup_{\be_n\in\Sigma_\beta^n}\{x\in I_{n,\beta}(\be_n):|T_\beta^nx-h(x)|<\psi(n)\},\]
	and since each set $\{x\in I_{n,\beta}(\be_n):|T_\beta^nx-h(x)|<\psi(n)\}$ is contained in an interval of length $\asymp\beta^{-n}\psi(n)$ (see Lemma \ref{l:length}), by the definition of Hausdorff $f$-measure,
	\[\hm^f\big(W_1(\psi,h)\big)\ll\limsup_{N\to\infty}\sum_{n=N}^{\infty}\sum_{\be\in\Sigma_\beta^n}g\big(\beta^{-n}\psi(n)\big)\asymp\limsup_{N\to\infty}\sum_{n=N}^{\infty}\beta^ng\big(\beta^{-n}\psi(n)\big)=0,\]
	where we use $\#\Sigma_\beta^n\asymp\beta ^n$ in the second inequality.

	Suppose that $\sum_{n=1}^\infty\beta^ng\big(\beta^{-n}\psi(n)\big)=\infty$. If $\lim_{r\to 0^+} g(r)/r<\infty$, then we have
	\[r\asymp g(r)\quad\text{for all small $r>0$}.\]
	This implies that $\hm^g=c\lm^1$ for some $c>0$ and
	\[\sum_{n=1}^\infty\beta^ng\big(\beta^{-n}\psi(n)\big)=\infty\quad\Longleftrightarrow\quad \sum_{n=1}^\infty\psi(n)=\infty.\]
	Therefore,
	\[\hm^g([0,1])\ge \hm^g\big(W_1(\psi,h)\big)\ge\hm^g\big(\tilde W_1(\psi,h)\big) =c\lm^1\big(\tilde W_1(\psi,h)\big)=c=\hm^g([0,1]),\]
	as desired.

	If $\lim_{r\to 0^+} g(r)/r=\infty$, then we have
	\begin{equation}\label{eq:r<g(r)}
		r\le g(r)\quad\text{for all small $r>0$}.
	\end{equation}
	By Lemma \ref{l:onedim},
	\begin{equation}\label{eq:fullmeasure}
		\lm^1\bigg(\bigcap_{N=1}^\infty\bigcup_{n=N}^\infty\bigcup_{\be_n\in\Lambda_\beta^n}\big\{x\in I_{n,\beta}(\be_n):|T_\beta^nx-h(x)|<\beta^ng\big(\beta^{-n}\psi(n)\big)\big\}\bigg)=1.
	\end{equation}
	Since $\beta^{-n}\psi(n)\to 0$ as $n\to\infty$, for large $n$ by \eqref{eq:r<g(r)} we have
	\[\psi(n)<\beta^ng(\beta^{-n}\psi(n)).\]
	Therefore, for large $n$ and for any full sequence $\be_n\in\Lambda_\beta^n$,
	\[\big\{x\in I_{n,\beta}(\be_n):|T_\beta^nx-h(x)|<\psi(n)\big\}\subset\big\{x\in I_{n,\beta}(\be_n):|T_\beta^nx-h(x)|<\beta^ng\big(\beta^{-n}\psi(n)\big)\big\}.\]
	By Lemma \ref{l:length}, the left set contains an interval $J$ with $|J|\asymp\beta^{-n}\psi(n)$ and the right set is contained in an interval $I$ with $|I|\asymp g(\beta^{-n}\psi(n))$. Moreover, we have
	\[\hm^g(J)=g(|J|)\asymp g\big(\beta^{-n}\psi(n)\big)\asymp |I|.\]
	It follows from \eqref{eq:fullmeasure} and Theorem \ref{t:weaken} (1) that
	\[\hm^g\big(\tilde W_1(\psi,h)\big) =\hm^g([0,1]),\]
	which implies that
	\[\hm^g\big(W_1(\psi,h)\big) =\hm^g([0,1])\]
	since $\tilde W_1(\psi,h)\subset W_1(\psi,h)$.
\end{proof}
\section{Proof of Theorem \ref{t:multiplicative} (2)}\label{s:multiplicative}
In this section, we do not assume $\beta$ is integral.
\subsection{Proof of Theorem \ref{t:multiplicative} (2): Convergence part}
We start with two lemmas that will be used later. The following lemma by Hussain and Simmons \cite{HS18} strengthens \cite[Lemma 1]{BD78} by showing that the diameter of the balls in a  covering of a hyperboloid can be made sufficiently large.
\begin{lem}[{\cite[Lemma 2.2]{HS18}}]\label{l:HScovering}
	Let $0<\delta\le 1$. For any $\ba=(a_1,\dots,a_d)\in[0,1]^d$ and $s\in (d-1,d)$, the set
	\[H_d(\ba,\delta):=\bigg\{\bx\in[0,1)^d:\prod_{i=1}^{d}|x_i-a_i|<\delta\bigg\}\]
	has a covering $\cb$ by $d$-dimensional balls $B$ such that
	\[|B|\ge \delta\quad\text{and}\quad\sum_{B\in\cb}|B|^s\ll\delta^{s-d+1}.\]
\end{lem}
The above lemma does not precisely correspond to \cite[Lemma 2.2]{HS18}, where the assumption  $\ba=(0,\dots,0)$ is made. However, the idea still applicable with a translation. The next lemma indicates that the sets under consideration can be regarded as the inverse of some hyperboloid sets $H_d(\ba,\delta)$.

\begin{lem}\label{l:inversehd}
	Let $L$ be the Lipschitz constant of $\bm h$. For any $n$ with $\beta_d^n\ge2L$ and any $\be_n^i\in\Sigma_{\beta_i}^n$ with $1\le i\le d$, there exists a point $\ba=(a_1,\dots,a_d)\in[0,1]^d$ depending on $\be_n^i$ ($1\le i\le d$) such that
	\begin{align}
		&\bigg\{\bx\in\prod_{i=1}^{d}I_{n,\beta_i}(\be_n^i):\prod_{i=1}^d|T_{\beta_i}^nx_i-h_i(x_i)|<\psi(n)\bigg\}\\
		\subset & \Big(T_{\beta_1}^n\times\cdots\times T_{\beta_d}^n|_{\prod_{i=1}^{d}I_{n,\beta}(\be_n^i)}\Big)^{-1}H_d\big(\ba,2^d\psi(n)\big).
	\end{align}
\end{lem}
\begin{proof}
	Note that for $1\le i\le n$, $T_{\beta_i}^n|_{I_{n,\beta_i}(\be_n^i)}$ is in general not onto $[0,1)$. However, $T_{\beta_i}^n|_{I_{n,\beta_i}(\be_n^i)}$ is a linear function with slope $\beta_i^n$ and maps the left endpoint of $I_{n,\beta_i}(\be_n^i)$ to $0$. With these properties in mind, we extend $I_{n,\beta_i}(\be_n^i)$ to a closed interval $I_n^{(i)}$ with length $\beta_i^{-n}$ and the same left endpoint as $I_{n,\beta_i}(\be_n^i)$. Then, there exist a map $G_i:I_n^{(i)}\mapsto [0,1]$ and a constant $c_i\ge 0$ such that
	\[G_i|_{I_{n,\beta_i}(\be_n^i)}=T_{\beta_i}^n|_{I_{n,\beta_i}(\be_n^i)}, \quad G_i(x_i^*)=0,\quad G_i(y_i^*)=1,\]
	and
	\[ G_i(x_i)=\beta_i^nx_i+c_i \text{ for any $x_i\in I_n^{(i)}$},\]
	where $x_i^*$ and $y_i^*$ are, respectively, the left and right endpoints of $I_n^{(i)}$.

	Since $G_i(x_i^*)-h_i(x_i^*)=0-h_i(x_i^*)\le 0$ and $G_i(y_i^*)-h_i(y_i^*)=1-h_i(y_i^*)\ge 0$, there exists a point $z_i\in I_n^{(i)}$ such that
	\[G_i(z_i)-h_i(z_i)=0\quad\Longleftrightarrow\quad G_i(z_i)=h_i(z_i).\]
	Let $\ba=(h_1(z_1),\dots,h_d(z_d))$. Suppose that
	\[\bx\notin \Big(T_{\beta_1}^n\times\cdots\times T_{\beta_d}^n|_{\prod_{i=1}^{d}I_{n,\beta_i}(\be_n^i)}\Big)^{-1}H_d\big(\ba,2^d\psi(n)\big),\]
	or equivalently
	\[\bx\in\prod_{i=1}^{d}I_{n,\beta_i}(\be_n^i)\quad\text{but}\quad (T_{\beta_1}^n\times\cdots\times T_{\beta_d}^n)\bx\notin H_d\big(\ba,2^d\psi(n)\big).\]
	Using $G_i|_{I_{n,\beta_i}(\be_n^i)}=T_{\beta_i}^n|_{I_{n,\beta_i}(\be_n^i)}$ and $G_i(z_i)=h_i(z_i)$, we have
	\begin{equation}\label{eq:2dpsi}
		2^d\psi(n)<\prod_{i=1}^d|G_i(x_i)-h_i(z_i)|=\prod_{i=1}^d|G_i(x_i)-G_i(z_i)|=\prod_{i=1}^d\beta_i^n|x_i-z_i|.
	\end{equation}
	On the other hand, since $\beta_d^n\ge 2L$ we get $\beta_i^n-L\ge \beta_i^n/2$ for all $1\le i\le d$. Hence,
	\[\begin{split}
		\prod_{i=1}^{d}|T_{\beta_i}^nx_i-h_i(x_i)|&=\prod_{i=1}^{d}|G_i(x_i)-h_i(x_i)|=\prod_{i=1}^{d}|G_i(x_i)-G_i(z_i)+h_i(z_i)-h_i(x_i)|\\
		&\ge \prod_{i=1}^{d}\big(|G_i(x_i)-G_i(z_i)|-|h_i(z_i)-h_i(x_i)|\big)\\
		&\ge \prod_{i=1}^{d}\big(\beta_i^n|x_i-z_i|-L|z_i-x_i|\big)\ge \prod_{i=1}^{d}\beta_i^n/2|x_i-z_i|\\
		&>\psi(n),
	\end{split}\]
	where the last inequality follows from \eqref{eq:2dpsi}. This is means that
	\[\bx\notin\bigg\{\bx\in\prod_{i=1}^{d}I_{n,\beta_i}(\be_n^i):\prod_{i=1}^d|T_{\beta_i}^nx_i-h_i(x_i)|<\psi(n)\bigg\},\]
	which completes the proof.
\end{proof}
We now move on to the task of proving the convergence part of Theorem \ref{t:multiplicative}. The proof follows the ideas in \cite[\S 2.1]{HS18} and \cite[Appendix A]{LLVZ23} closely.
\begin{prop}
	Let $1<\beta_1\le\cdots\le \beta_d$. Let $f$ be a dimension function such that $(d-1)\prec f$ and $f\preceq s$ for some $s\in(d-1,d)$. Then,
	\[\sum_{n=1}^{\infty}\beta_d^{dn}\psi(n)^{-d+1}f\big(\beta_d^{-n}\psi(n)\big)<\infty\quad\Longrightarrow\quad\hm^f\big(W_d^\times(\psi,\bm h)\big)=0.\]
\end{prop}
\begin{proof}
	By Lemma \ref{l:inversehd},
	\[\begin{split}
		W_d^\times(\Psi,\bm h)&=\bigcap_{N=1}^\infty\bigcup_{n=N}^\infty\bigcup_{\be_n^1\in\Sigma_{\beta_1}^n}\cdots\bigcup_{\be_n^d\in\Sigma_{\beta_d}^n}\bigg\{\bx\in\prod_{i=1}^{d}I_{n,\beta_i}(\be_n^i):\prod_{i=1}^d|T_{\beta_i}^nx_i-h_i(x_i)|<\psi(n)\bigg\}\\
		&\subset \bigcap_{N=1}^\infty\bigcup_{n=N}^\infty\bigcup_{\be_n^1\in\Sigma_{\beta_1}^n}\cdots\bigcup_{\be_n^d\in\Sigma_{\beta_d}^n}\Big(T_{\beta_1}^n\times\cdots\times T_{\beta_d}^n|_{\prod_{i=1}^{d}I_{n,\beta_i}(\be_n^i)}\Big)^{-1}H_d\big(\ba,2^d\psi(n)\big),
	\end{split}\]
	where $\ba$ depends on $\be_n^i$, $1\le i\le d$.
	Apply Lemma \ref{l:HScovering} with $\delta=2^d\psi(n)$, for any $s\in(d-1,d)$ there exists a covering $\cb_n$ of the hyperboloid $H_d(\ba,2^d\psi(n))$ by balls $B$ such that
	\begin{equation}\label{eq:coverofHd}
		|B|\ge 2^d\psi(n)\quad\text{and}\quad\sum_{B\in\cb_n}|B|^s\ll\psi(n)^{s-d+1}.
	\end{equation}
	It follows that for any $\be_n^i\in\Sigma_{\beta_i}^n$ with $1\le i\le d$,
	\[\Big(T_{\beta_1}^n\times\cdots\times T_{\beta_d}^n|_{\prod_{i=1}^{d}I_{n,\beta_i}(\be_n^i)}\Big)^{-1}H_d\big(\ba,2^d\psi(n)\big)\subset \bigcup_{B\in\cb_n}\Big(T_{\beta_1}^n\times\cdots\times T_{\beta_d}^n|_{\prod_{i=1}^{d}I_{n,\beta_i}(\be_n^i)}\Big)^{-1}B.\]
	For any $B\in\cb_n$, $(T_{\beta_1}^n\times\cdots\times T_{\beta_d}^n|_{\prod_{i=1}^{d}I_{n,\beta_i}(\be_n^i)})^{-1}B$ is contained in a hyperrectangle with sidelengths
	\[\beta_1^{-n}|B|,\dots,\beta_d^{-n}|B|.\]
	Since $1<\beta_1\le\cdots\le \beta_d$, such hyperrectangle can be covered by
	\[\asymp \prod_{i=1}^{d}\frac{\beta_i^{-n}|B|}{\beta_d^{-n}|B|}=\prod_{i=1}^{d}\beta_i^{-n}\beta_d^{n}\]
	balls of radius $\beta_d^{-n}|B|$.
	By the definition of Hausdorff $f$-measure,
	\[\begin{split}
		\hm^f\big(W_d^\times(\psi,\bm h)\big)&\ll\limsup_{N\to\infty}\sum_{n=N}^\infty\sum_{\be_n^1\in\Sigma_{\beta_1}^n}\cdots\sum_{\be_n^d\in\Sigma_{\beta_d}^n}\sum_{B\in\cb_n}\bigg(\prod_{i=1}^{d}\beta_i^{-n}\beta_d^{n}\bigg)f(\beta_d^{-n}|B|)\\
		&\asymp\limsup_{N\to\infty}\sum_{n=N}^\infty\beta_d^{dn}\sum_{B\in\cb_n}f(\beta_d^{-n}|B|),
	\end{split}\]
	where we use $\#\Sigma_{\beta_i}^n\asymp\beta_i^n$ in the last inequality. Since $f\preceq s$ and $|B|>2^d\psi(n)>\psi(n)$, for large $n$ we have
	\[\frac{f(\beta_d^{-n}|B|)}{(\beta_d^{-n}|B|)^s} \le\frac{f\big(\beta_d^{-n}\psi(n)\big)}{\big(\beta_d^{-n}\psi(n)\big)^s} \quad\Longrightarrow\quad f(\beta_d^{-n}|B|)\le |B|^s\psi(n)^{-s}f\big(\beta_d^{-n}\psi(n)\big).\]
	Therefore, by the last inequality and \eqref{eq:coverofHd},
	\[\begin{split}
		\hm^f\big(W_d^\times(\psi,\bm h)\big)&\ll\limsup_{N\to\infty}\sum_{n=N}^\infty\beta_d^{dn}\sum_{B\in\cb_n}f(\beta_d^{-n}|B|)\\
		&\le\limsup_{N\to\infty}\sum_{n=N}^\infty\beta_d^{dn}\sum_{B\in\cb_n}|B|^s\psi(n)^{-s}f\big(\beta_d^{-n}\psi(n)\big)\\
		&\ll\limsup_{N\to\infty}\sum_{n=N}^\infty\beta_d^{dn}\psi(n)^{s-d+1}\psi(n)^{-s}f\big(\beta_d^{-n}\psi(n)\big)\\
		&=\limsup_{N\to\infty}\sum_{n=N}^\infty\beta_d^{dn}\psi(n)^{-d+1}f\big(\beta_d^{-n}\psi(n)\big)=0,
	\end{split}\]
	where the last inequality follows from  $\sum_{n=1}^\infty\beta_d^{dn}\psi(n)^{-d+1}f\big(\beta_d^{-n}\psi(n)\big)<\infty$.
\end{proof}
\subsection{Proof of Theorem \ref{t:multiplicative} (2): Divergence part}
The proof will proceed using the following ``Slicing Lemma" for Hausdorff measures (see, for example, \cite[Lemma 4]{BV06b}).
\begin{lem}[Slicing Lemma]\label{l:slicing}
	Let $0<k<d$, and let $g$ be a dimension function and let $f(r)=r^kg(r)$. Let $A$ be a Borel subset of $[0,1)^d$ and suppose that the set
	\[\{\bx_1\in [0,1)^k:\hm^g(\{\bx_2\in [0,1)^{d-k}:(\bx_1,\bx_2)\in A\})=\infty\}\]
	has positive $\hm^k$-measure. Then $\hm^f(A)=\infty$.
\end{lem}
We now establish the complementary divergence part for the Hausdorff measure of $W_d^\times(\psi,\bm h)$.
\begin{prop}\label{p:lowmultiplicative}
	Let $1<\beta_1\le\cdots\le \beta_d$. Let $f$ be a dimension function such that $(d-1)\prec f\preceq s$ for some $s\in(d-1,d)$. Then,
	\[\sum_{n=1}^{\infty}\beta_d^{dn}\psi(n)^{-d+1}f\big(\beta_d^{-n}\psi(n)\big)=\infty\quad\Longrightarrow\quad\hm^f\big(W_d^\times(\psi,\bm h)\big)=\hm^f([0,1]^d).\]
\end{prop}
\begin{proof}
	Notice that
	\begin{align}
		W_d(\psi,\bm h)&\supset\{\bx\in[0,1)^d:|T_{\beta_d}^nx_d-h_d(x_d)|<\psi(n)\text{ for i.m.\,$n$}\}\notag\\
		&=[0,1)^{d-1}\times \{x_d\in[0,1):|T_{\beta_d}^nx_d-h_d(x_d)|<\psi(n)\text{ for i.m.\,$n$}\}.\label{eq:product}
	\end{align}
	With a slightly abuse of notation, write the bottom-right $\limsup$ set as $W_1(\psi,h)$. Let $g(r)=f(r)/r^{d-1}$. By the assumption $d-1\prec f\preceq s$ for some $s\in (d-1,d)$, $g$ is a dimension function such that $g\prec 1$. Since
	\[\sum_{n=1}^{\infty}\beta_d^ng\big(\beta_d^{-n}\psi(n)\big)=\sum_{n=1}^{\infty}\beta_d^n\cdot\frac{f\big(\beta_d^{-n}\psi(n)\big)}{\big(\beta_d^{-n}\psi(n)\big)^{d-1}}=\infty,\]
	by Theorem \ref{t:multiplicative} (1),
	\[\hm^g\big(W_1(\psi,h)\big)=\infty.\]
	Applying the Slicing Lemma with $k=d-1$, we deduce from \eqref{eq:product} that
	\[\hm^f\big(W_d(\psi,\bm h)\big)=\infty=\hm^f([0,1]^d).\qedhere\]
\end{proof}

\end{document}